\documentclass[preprint,10pt,3p]{elsarticle}
\usepackage{verbatim}
\usepackage{amsmath,amsthm,amsfonts,amssymb}
\usepackage{cases}
\usepackage{hyperref}
\usepackage{multirow}
\usepackage{amsmath,amsfonts,amssymb}
\usepackage[normalem]{ulem}
\usepackage{xcolor,sectsty}
\usepackage{natbib}
\usepackage{stackengine}
\newcommand{\core}{\scriptsize\mbox{\tiny\textcircled{\#}}}
\newcommand{\ep}{\scriptsize\mbox{\tiny\textcircled{$\dagger$}}}
\definecolor{astral}{RGB}{46,116,181}
\subsectionfont{\color{astral}}
\sectionfont{\color{astral}}

\newtheorem{theorem}{Theorem}[section]
\newtheorem{lemma}{Lemma}[section]
\newtheorem{corollary}{Corollary}[section]
\newtheorem{proposition}{Proposition}[section]
\newtheorem{definition}{Definition}[section]
\newtheorem{example}{Example}[section]
\newtheorem{remark}{Remark}[section]
\definecolor{darkslategray}{rgb}{0.18, 0.31, 0.31}
\definecolor{warmblack}{rgb}{0.0, 0.26, 0.26}
\journal{arxiv}

\newcommand{\R}{{\mathbb R}}

\newcommand{\mc}[1]{\mathcal {#1}}
\newcommand{\dg}{{\dagger}}
\newcommand{\n}{{*_N}}

\newcommand{\m}{{*_M}}

\def\R{{\cal{R}}}

\def\N{{\cal{N}}}

\newcommand{\rra}[1]{\mathrm{rshrank}({#1})}  
\newcommand{\ra}[1]{\mathrm{rank}({#1})}  
\newcommand{\ind}[1]{\mathrm{ind}({#1})}  

\newcommand{\T}{{\mathrm T}}
\newcommand{\D}{{\mathrm D}}

\numberwithin{equation}{section}

\begin{document}

\begin{frontmatter}

\title{ \textcolor{warmblack}{\bf Core and Core-EP Inverses of Tensors}}

\author{Jajati Kesahri Sahoo$^a$,   Ratikanta Behera$^b$, Predrag S.  Stanimirovi{\'c}$^c$, Vasilios N. Katsikis$^d$, Haifeng Ma$^e$}

\address{$^a$Department of Mathematics,\\
BITS Pilani, K.K. Birla Goa Campus, Goa, India
\\\textit{E-mail}: \texttt{jksahoo\symbol{'100}goa.bits-pilani.ac.in}
\vspace{.3cm}\\
$^b$Department of Mathematics and Statistics,\\
Indian Institute of Science Education and Research Kolkata,\\
 Nadia, West Bengal, India.\\
\textit{E-mail}: \texttt{ratikanta@iiserkol.ac.in}
\vspace{.3cm}\\
$^c$University of Ni\v s, Faculty of Science and Mathematics,\\
Department of Computer Science, \\
Vi\v segradksa 33, 18000 Ni\v s, Serbia;\\
\textit{E-mail}: \texttt{pecko@pmf.ni.ac.rs}\vspace{.3cm}\\
$^d$National and Kapodistrian University of Athens,\\
Sofokleous 1 Street, 10559 Athens, Greece;\\
\textit{E-mail}: \texttt{vaskatsikis@econ.uoa.gr}
\vspace{.3cm}\\
$^e$Harbin Normal University, School of Mathematical Science\\
Harbin 150025, China.\\
\textit{E-mail}: \texttt{haifengma@aliyun.com}
}

\begin{abstract}
\textcolor{warmblack}{
Specific definitions of the core and core-EP inverses of complex tensors are introduced.
Some characterizations, representations and properties of the core and core-EP inverses are investigated.
The results are verified using specific algebraic approach, based on proposed definitions and previously verified properties.
The approach used here is new even in the matrix case.
}
\end{abstract}


\end{frontmatter}

\section{Introduction, background and motivation}\label{sec1}

The core inverse for matrices was introduced by Baksalary \cite{baks,baks1}.
Several properties of the core inverse and interconnections with other generalized inverses were further studied by many researchers \cite{baks,kurata,rakic,wang2015,xu2017}.
Perturbation bounds of the core inverse were considered in \cite{Ma18}.
An extension of the core inverse in the finite dimensional space to the set of bounded Hilbert space operators were proposed in \cite{rakicAMC}.
Raki\'c {\em et al.} in \cite{rakic} extended the notion of the core inverse from the complex matrix space to an arbitrary $*$-ring.

Prasad, and Raj in \cite{MD} introduced a bordering method for computing the core-EP inverse by relating this generalized inverse with a suitable bordered matrix.
Some further properties of the core-EP inverse were investigated in \cite{ferreyra}.
Prasad, Raj, and Vinay in \cite{CoreIter} introduced an iterative method to approximate the core-EP inverse.
Three limit representations of the core-EP inverse were proposed in \cite{ZhouLimit}.
The notion of the core inverse was generalized to the core-EP inverse.
Many characterizations of the core-EP inverse with other inverses are discussed in \cite{gao2018,prasad}.
The core-EP was extended to rectangular matrices in \cite{ferreyra,gao2019}.
Some new characterizations, representations and the perturbation bounds of the core-EP and the weighted core-EP were investigated in \cite{Ma19,Ma191}.

\smallskip
The tensor inversion and generalized inversion based on different tensor products have been a frequent topic for investigation in the available literature.
The representations and properties of the tensor inverse were considered in \cite{BrazellT}.
Further, Liang et al. in \cite{LiangT} investigated necessary and sufficient conditions for the invertibility of a given tensor.
Representations and properties of the Moore-Penrose inverse of tensors were derived in \cite{Behera} and \cite{SunT}.
Representations for the weighted Moore-Penrose inverse of tensors were considered in \cite{JiT}.
Panigrahy and Mishra in \cite{Panigrahy} investigated the Moore-Penrose inverse of the Einstein product of two tensors.
Ji and Wei \cite{Ji18} investigated the Drazin inverse of even-order square tensors under the Einstein product.

So, it is observable that the core and core-EP inverse are not investigated in the tensor case, so far.
The targets of our research, in the present article, are core and core-EP inverses of tensors.
Corresponding definitions are introduced and main properties are investigated.

\smallskip
Main contributions of this manuscript can be summarized as follows.

\noindent (1) Definitions of the core and core-EP inverses of tensors are introduced.

\noindent (2) Some characterizations and properties of the core and core-EP inverses are investigated.

\noindent (3) The results are verified using one specific algebraic approach, which is new even in the matrix case.

\smallskip
The rest of the paper is organized as follows.
Some necessary notations, useful known results and definitions are presented in Section \ref{Preliminaries}.
Definition, basic properties and representations of the core inverse are considered in Section \ref{SecCore}.
The core inverse of the sum of two tensors is investigated in Subsection \ref{SubsecCoreSUm}.
Section \ref{SecCoreEP} is aimed to the core-EP inverse of a square tensor.
Some properties, representations and characterizations of the core-EP inverse are given.
Illustrative numerical examples are given in Section \ref{SecExamples}.
Concluding remarks are stated in Section \ref{SecConclusion}.

\section{Preliminaries}\label{Preliminaries}



For convenience, we first briefly explain some of the terminologies which will be used here on wards.
We refer to $\mathbb{C}^{I_1\times\cdots\times I_N}$ (resp. $\mathbb{R}^{I_1\times\cdots\times I_N}$) as the set of order $N$ complex (resp. real) tensors.
Particularly, a matrix is a second order tensor, and a vector is a first order tensor.
Each entry of a tensor $\mc{A} \in \mathbb{C}^{I_1\times\cdots\times I_N}$ is denoted by $\mc{A}_{i_1...i_N}$.
Note that throughout the paper, tensors are represented in calligraphic letters like  $\mc{A}$, and the notation $\mc{A}_{i_1...i_N}$ represents the scalars.
The Einstein product (\cite{ein}) $ \mc{A}\n\mc{B}$ of tensors
$\mc{A} \in \mathbb{C}^{I_1\times\cdots\times I_N \times K_1 \times\cdots\times K_N }$ and $\mc{B} \in \mathbb{C}^{K_1\times\cdots\times K_N \times J_1 \times\cdots\times J_M}$ is defined
by the operation $\n$ via
\begin{equation}\label{Eins}
(\mc{A}\n\mc{B})_{i_1...i_Nj_1...j_M}
=\displaystyle\sum_{k_1...k_N}\mc{A}_{{i_1...i_N}{k_1...k_N}}\mc{B}_{{k_1...k_N}{j_1...j_M}}\in \mathbb{C}^{I_1\times\cdots\times I_N \times J_1 \times\cdots\times J_M }.
\end{equation}
Specifically, if $\mc{B} \in \mathbb{C}^{K_1\times\cdots\times K_N}$, then $\mc{A}\n\mc{B} \in \mathbb{R}^{I_1\times\cdots\times I_N}$ and
\begin{equation*}\label{Einsb}
(\mc{A}\n\mc{B})_{i_1...i_N} = \displaystyle\sum_{k_1...k_N} \mc{A}_{{i_1...i_N}{k_1...k_N}}\mc{B}_{{k_1...k_N}}.
\end{equation*}
The Einstein product is discussed in the area of continuum mechanics  \cite{ein} and the theory of relativity \cite{lai}.
Further, the sum of two tensors $\mc{A}, ~\mc{B}\in \mathbb{C}^{I_1\times\cdots\times I_N \times K_1 \times\cdots\times K_N }$ is
\begin{equation}\label{Eins1}
(\mc{A} + \mc{B})_{i_1...i_Nk_1...k_N} =\mc{A}_{{i_1...i_N}{k_1...k_N}} + \mc{A}_{{i_1...i_N}{k_1...k_N}}.
\end{equation}
For a tensor $\mc{A}\!=\!(\mc{A}_{{i_1}...{i_M}{j_1}...{j_N}})  \in \mathbb{R}^{I_1\times\cdots\times I_M \times J_1 \times\cdots\times J_N},$ the tensor
$\mc{A}^\T\!=\!(\mc{A}_{{j_1}...{j_N}{i_1}...{i_M}})\in \mathbb{R}^{J_1\times\cdots\times J_N\times I_1 \times\cdots\times I_M}$ is the {\it transpose} of $\mc{A}$.
The conjugate transpose of a tensor
$\mathcal{A} = (\mathcal{A}_{  i_{1}\ldots i_{M}j_{1} \ldots j_{N}}) \in \mathbb{C}^{\mathbf{I}_{1} \times \cdots \times \mathbf{I}_{M} \times \mathbf{J}_{1} \times \cdots \times \mathbf{J}_{N}}$
is denoted by $\mathcal{A}^{\ast}$ and elementwise defined as
$(\mathcal{A}^{\ast})_{ j_{1} \ldots j_{N} i_{1} \ldots i_M}=(\overline{\mathcal{A}})_{ i_{1} \ldots i_{M} j_{1} \ldots j_N}
\in \mathbb{C}^{\mathbf{J}_{1} \times \cdots \times \mathbf{J}_{N} \times \mathbf{I}_{1} \times \cdots \times \mathbf{I}_{M}}$
where the over-line means the conjugate operator.
 Further, a tensor $\mc{O}$ denotes the {\it zero tensor} if  all the entries are zero.
 A tensor $\mc{A}\in \mathbb{R}^{I_1\times\cdots\times I_N \times I_1 \times\cdots\times I_N}$ is {\it symmetric}
 if  $\mc{A}=\mc{A}^\T,$ {\it skew-symmetric} if $\mc{A}= - \mc{A}^\T$, and {\it orthogonal} if $\mc{A}\m\mc{A}^\T= \mc{A}^\T\n \mc{A}=\mc{I}$.
 A tensor $\mc{A}\in \mathbb{C}^{I_1\times\cdots\times I_N \times I_1 \times\cdots\times I_N}$  is {\it idempotent}  if $\mc{A}\n \mc{A}= \mc{A}$ and {\it tripotent} if $\mc{A}^3= \mc{A}$. 
Further, a tensor $\mc{A}\in \mathbb{C}^{I_1\times\cdots\times I_N \times I_1 \times\cdots\times I_N}$ is called {\it Hermitian idempotent } if $\mc{A}^2=\mc{A}=A^\dagger$ or  $\mc{A}^2=\mc{A}=A^{*}.$

The additional notation
$
I(N)=I_1\times \cdots \times I_N
$
will be useful in increasing the efficiency of the presentation.
Accordingly, the tensor $\mathcal{A} \in \mathbb{C}^{I_{1} \times \cdots \times I_{{M}} \times J_{1} \times \cdots \times J_N}$ will be denoted in a simpler form as $\mathcal{A} \in \mathbb{C}^{I(M) \times J(N)}$.
Tensors of the form $\mathbb{C}^{I(N)\times J(N)}$ are known as even-order tensors.
Following the terminology from \cite{Ji18}, even-order tensors of the shape $\mathbb{C}^{I(N)\times I(N)}$ will be termed as even-order square tensors, or simply {\em square} tensors. 

\begin{definition}
Let $\mc{A}\in \mathbb{C}^{I(N)\times I(N)}$. 
Then the tensor $\mc{A}$ is called EP if $\mc{A}\n\mc{A}^\dagger=\mc{A}^\dagger\n\mc{A}.$
\end{definition}

\begin{definition}
A tensor $\mc{A}\in\mathbb{C}^{I(M) \times J(N)}$ is called partial isometry if $\mc{A}\n\mc{A}^*\m\mc{A}=\mc{A}$ or $\mc{A}^\dagger=\mc{A}^*.$
\end{definition}


The definition of a diagonal tensor follows from \cite{SunT}. 
\begin{definition} {\em \cite{SunT}}
A tensor with entries $\mc{D}_{{i_1}\ldots{i_N}{j_1}\ldots{j_N}}$
  is called a {\it diagonal tensor} if $\mc{D}_{{i_1}\ldots{i_N}{j_1}\ldots{j_N}} = 0$ for $(i_1,\ldots,i_N) \neq (j_1,\ldots,j_N).$
\end{definition}

We  recall the definition of an identity tensor below.

\begin{definition} {\em \cite[Definition 3.13]{BrazellT}}
A tensor with entries
     $\mc{I}_{i_1i_2 \cdots i_Nj_1j_2\cdots j_N} = \prod_{k=1}^{N} \delta_{i_k j_k}$,
   where
\begin{numcases}
{\delta_{i_kj_k}=}
  1, &  $i_k = j_k$,\nonumber
  \\
  0, & $i_k \neq j_k $,\nonumber
\end{numcases}
 is  called a {\it  unit tensor or identity tensor}.
\end{definition}

\begin{definition} {\em \cite[Definition 2.1]{stan}}
The range and null space of a tensor $\mc{A}\in \mathbb{R}^{I(M) \times J(N)}$ are defined as per the following:
$$
\mathfrak{R}(\mc{A}) = \left\{\mc{A}\n\mc{X}:~\mc{X}\in\mathbb{R}^{{J_1}\times\ldots\times{J_N}}\right\}\mbox{ and }
\mc{N}(\mc{A})=\left\{\mc{X}:~\mc{A}\n\mc{X}=\mc{O}\in\mathbb{R}^{{I_1}\times \ldots \times {I_M}}\right\},
$$
where $\mathcal{O}$ is an appropriate zero tensor.
\end{definition}

The index of a tensor $\mc{A}\in\mathbb{R}^{I(N)\times I(N)}$ is defined as the smallest positive integer $k,$ such that $\mathfrak{R}(\mc{A}^k)=\mathfrak{R}(\mc{A}^{k+1}).$
If $k=1,$ then the tensor is called index one or group, or core tensor.
The index of a square tensor $\mc{A}$ is denoted as $\ind{\mc{A}}$.


A tensor $\mathcal{A} \in \mathbb{C}^{{I}_{1} \times \cdots \times {I}_{{N}} \times {I}_{1} \times \cdots \times {I}_{{N}}}$ is invertible
if there exists a tensor $\mathcal{X}$ such that $\mathcal{A}\ast_{{N}}\mathcal{X} = \mathcal{X}\ast_{{N}}\mathcal{A} = \mathcal{I}$.
In this case, $\mathcal{X}$ is called the inverse of $\mathcal{A}$ and denoted by $\mathcal{A}^{-1}$.

\begin{lemma} {\em \cite[Lemma 2.2.]{stan}}\label{range-stan}
Let  $\mc{A}\in \mathbb{R}^{{I_1}\times \cdots\times {I_M}\times {J_1}\times\cdots\times {J_N}}$,
$\mc{B}\in \mathbb{R}^{{I_1}\times \cdots\times {I_M}\times {K_1}\times\cdots\times {K_L}}.$
Then $\mathfrak{R}(\mc{B})\subseteq\mathfrak{R}(\mc{A})$ if and only if there exists
$\mc{U}\in \mathbb{R}^{{J_1}\times \ldots\times{J_N}\times {K_1}\times\ldots\times{K_L}}$ such that
$\mc{B}=\mc{A}\n\mc{U}.$
\end{lemma}

Let $\mc{A}\in \mathbb{C}^{I(M) \times J(N)},$ $\mc{X}\in \mathbb{C}^{J(N) \times I(M)}$ and consider the following conditions
\[
\aligned
& (1^T)\quad  \mathcal{A}\ast_{{N}}\mathcal{X}\ast_{{N}}\mathcal{A} = \mathcal{A};\qquad\quad \,
(2^T)\quad  \mathcal{X}\ast_{{N}}\mathcal{A}\ast_{{N}}\mathcal{X} = \mathcal{X};\\
&(3^T)\quad (\mathcal{A}\ast_{{N}}\mathcal{X})^{\ast} = \mathcal{A}\ast_{{N}}\mathcal{X};\qquad
(4^T)\quad (\mathcal{X}\ast_{{N}}\mathcal{A})^{\ast} = \mathcal{X}\ast_{{N}}\mathcal{A}.
\endaligned
\]
\begin{definition} {\em \cite[Definition 2.2]{SunT}}
The tensor $\mc{X}$ satisfying the tensor equations:
\begin{enumerate}
\item[] $(1^T)$ is called inner or generalized inverse of $\mc{A}$ and denoted by $\mc{A}^{(1^T)}.$
\item[] $(2^T)$ is called outer inverse of $\mc{A}$ and denoted by $\mc{A}^{(2)}.$
\item[] $(1^T)$ and $(2^T)$ is called reflexive generalized inverse of $\mc{A}$ and denoted by $\mc{A}^{(1^T,2^T)}.$
\item[] $(1^T)$--$(4^T)$ is called Moore-Penrose inverse of $\mc{A}$ and denoted by $\mc{A}^{\dagger}.$
\end{enumerate}
\end{definition}
We use the notation $\mc{A}\{1^T\}$, $\mc{A}\{1^T,2^T\}$ and $\mc{A}\{1^T,3^T\}$ respectively for the set of inner inverses, reflexive generalized inverses of $\mc{A}$, and $\{1,3\}$ inverses of $\mc{A}$.
Similarly, the sets $\mc{A}\{1^T,4^T\}$ and $\mc{A}\{1^T,2^T, 3^T\}$ are defined.
The group and Drazin inverse are defined as follows for an even-order square tensor $\mc{A}\in \mathbb{C}^{I(N) \times I(N)}$.

\begin{definition}
let $\mc{A}\in \mathbb{C}^{I(N) \times I(N)}$ be a core tensor.
A tensor $\mc{X}\in \mathbb{C}^{I(N) \times I(N)}$ satisfying
$$
(1^T) \ \mc{A}\n\mc{X}\n\mc{A}=\mc{A},\quad (2^T) \ \mc{X}\n\mc{A}\n\mc{X}=\mc{X},\quad (5^T)\ \mc{A}\n\mc{X}=\mc{X}\n\mc{A},
$$
is called the group inverse of $\mc{A}$, and it is denoted by $\mc{A}^{\#}.$
\end{definition}

\begin{definition} {\em \cite{Ji18}}
Let $\mc{A}\in \mathbb{C}^{I(N) \times I(N)}$ with ind$(\mc{A})=k.$
A tensor $\mc{X}\in \mathbb{C}^{I(N) \times I(N)}$ satisfying
$$
\quad ((1^k)^T)\  \mc{A}^{k+1}\n\mc{X}=\mc{A}^k, \quad (2^T) \ \mc{X}\n\mc{A}\n\mc{X}=\mc{X},\quad (5^T) \ \mc{X}\n\mc{A}=\mc{A}\n\mc{X}
$$
 is called the Drazin inverse of $\mc{A}$ and it is denoted by $\mc{A}^\D.$
\end{definition}
It is important to mention that the matrix equations corresponding to $(1)^T$,$(2)^T$,$(3)^T$,$(4)^T$,$(5)^T$,$(1^k)^T$ are denoted by $(1)$,$(2)$,$(3)$,$(4)$,$(5)$,$(1^k)$, respectively.

The orthogonal projection onto a subspace $\mathcal{R}(\mathcal{A})$ is denoted by ${P}_{\mathcal{R}(\mathcal{A})}$ and defined as
${P}_{\mathcal{R}(\mathcal{A})} = \mathcal{A}\ast_{{N}}\mathcal{A}^\dagger.$

One specific and useful algorithm for a proper matricization of an arbitrary tensor was proposed in \cite{stan}.

\begin{definition}\label{Defrsh} {\em \cite{stan}}
Let $I_1,\ldots ,I_M,K_1,\ldots ,K_N$ be given integers and $\mathfrak{I}, \mathfrak{K}$ are the integers defined as
\begin{equation}\label{IK}
\mathfrak{I}=I_{1}I_{2}\cdots I_{M}, \ \ \mathfrak{K}=K_{1}K_{2}\cdots K_{N}.
\end{equation}
The reshaping operation
\[
\mathrm{rsh}: \mathbb{C}^{I({M}) \times K({N})}\mapsto \mathbb C^{\mathfrak{I}\times \mathfrak{K}}
\]
transforms a tensor $\mathcal{A}\in \mathbb{C}^{I({M}) \times K({N})}$
into the matrix $A\in \mathbb C^{\mathfrak{I}\times \mathfrak{K}}$ using the {\em Matlab} function {\tt reshape} as follows:
\[
\mathrm{rsh}\left(\mathcal{A}\right)=A=\mathrm{reshape}(\mathcal{A},\mathfrak{I},\mathfrak{K}),\ \ \mathcal{A}\in \mathbb{C}^{I({M}) \times K({N})},\ A\in \mathbb C^{\mathfrak{I}\times \mathfrak{K}}.
\]

\noindent
The inverse reshaping of $A\in \mathbb C^{\mathfrak{I}\times \mathfrak{K}}$ is the tensor $\mathcal{A}\in \mathbb{C}^{I({M}) \times K({N})}$ defined by
$$
\aligned
\mathrm{rsh}^{-1}(A)&=\mathcal{A}=\mathrm{reshape}(A,I_1,\ldots ,I_M,K_1,\ldots ,K_N).
\endaligned
$$
\end{definition}

Also, an appropriate definition of the tensor rank, arising from the reshaping operation, was proposed in \cite{stan}.
\begin{definition}\label{DerTRank} {\em \cite{stan}}
Let $\mathcal{A}\in \mathbb{C}^{I({N}) \times K({N})}$ and
$A=reshape \left(\mathcal{A},\mathfrak{I}, \mathfrak{K}\right)=\mathrm{rsh}(\mathcal A)\in \mathbb C^{\mathfrak{I}\times \mathfrak{K}}$.
Then the tensor rank of $\mathcal{A}$, denoted by $\mathrm{rshrank}(\mathcal{A})$, is defined by $\rra{\mathcal{A}}=\ra{A}$.
\end{definition}

\section{Tensor core inverse}\label{SecCore}

Let $A$ be an $n\times n$ complex matrix of index $1$.
The core inverse $A^{\tiny\textcircled{\#}} \in \mathbb{C}^{n \times n}$ of $A\in \mathbb{C}^{n \times n}$ was introduced in \cite[Definition 1]{baks} as the matrix satisfying
\begin{equation}\label{DefCore}
AA^{\tiny\textcircled{\#}}=P_A, \qquad \mathfrak{R}(A^{\tiny\textcircled{\#}}) \subseteq  \mathfrak{R}(A),
\end{equation}
where $P_A~(= AA^{\dagger})$ denotes an orthogonal projection onto the range $\mathfrak{R}(A)$.
Various characterizations of the core inverse, including its correlations with another generalized inverses, were originated in \cite{baks}.
The results obtained in \cite{baks} were developed on the basis of the Hartwig and Spindelb\"{o}ck decomposition.
In \cite{rakicAMC}, the authors introduced an equivalent definition of the core inverse, by proving in Theorem 3.1. that \eqref{DefCore} is equivalent to
\begin{equation}\label{DefCoreH}
AA^{\tiny\textcircled{\#}}A=A,\ \ \R(A^{\tiny\textcircled{\#}})=\R(A),\ \ \N(A^{\tiny\textcircled{\#}})=\N(A^*).
\end{equation}
Analogous definition of the core inverse in a ring $R$ with involution was introduced in \cite{rakic}:
\begin{equation}\label{DefCoreR}
a a^{\tiny\textcircled{\#}}a=a,\ \ a^{\tiny\textcircled{\#}}R=aR,\ \ R a^{\tiny\textcircled{\#}}=R a^*.
\end{equation}
Finally, the core inverse in the set
$
\mathbb{C}^{CM}_n= \{ A \in \mathbb{C}^{n \times n} |\ {\rm rank} (A^2)= {\rm rank}(A) \}.
$
was also defined \cite{rakic}:
\begin{equation}\label{DefCoreCM}
A A^{\tiny\textcircled{\#}}A=A,\ \ A^{\tiny\textcircled{\#}}\, \mathbb{C}^{CM}_n=A\, \mathbb{C}^{CM}_n,\ \ \mathbb{C}^{CM}_n\, A^{\tiny\textcircled{\#}}=\mathbb{C}^{CM}_n\, A^*.
\end{equation}
In \cite[Theorem 3.1]{xu2017}, the authors introduced the following representation of the core inverse in a ring $R$ with involution:
\begin{equation}\label{DefCoreR1}
(a a^{\tiny\textcircled{\#}})^*=a a^{\tiny\textcircled{\#}},\ \ a^{\tiny\textcircled{\#}}a^2=a,\ \ a \left(a^{\tiny\textcircled{\#}}\right)^2=a^{\tiny\textcircled{\#}}.
\end{equation}

\smallskip
In this section, we will introduce the core inverse of tensors and investigate its properties and representations.
Also, several characterizations of the tensor core inverse as well as some relationships with other generalized inverses will be discussed.

A tensor $\mc{A}$ is said to be a {\em core invertible} or a {\em core tensor} if the core inverse of $\mc{A}$ exists.
For this purpose, we will say that a tensor $\mathcal{A}$ is a {\em core tensor} if its reshaping index is equal to $1$, i.e.,
$$\rra{\mathcal{A}}=\rra{\mathcal{A}^2}\Longleftrightarrow \ra{\mathrm{rsh}(\mathcal A)}=\ra{\mathrm{rsh}(\mathcal A)^2}.$$

In the present paper the tensor core inverse is introduced in Definition \ref{CoreDefT}, using the approach analogous to \eqref{DefCoreR1}.

\begin{definition}\label{CoreDefT}
Let $\mc{A}\in \mathbb{C}^{I(N) \times I(N)}$ be a given core tensor.
A tensor $\mc{X}\in \mathbb{C}^{I(N) \times I(N)}$ satisfying
\begin{enumerate}
    \item[\rm $(C1)$] $\mc{X}\n\mc{A}^2=\mc{A}$
    \item[\rm $(C2)$] $\mc{A}\n\mc{X}^2=\mc{X}$
    \item[\rm $(3^T)$] $(\mc{A}\n\mc{X})^*=\mc{A}\n\mc{X}$
\end{enumerate}
is called the core inverse of $\mc{A}$ and denoted by $\mc{A}^{\core}$.
\end{definition}

It is known that the core inverse $A^{\tiny\textcircled{\#}}$ of a square matrix $A$ is the unique $\{1,2\}$-inverse which satisfies $\R(A^{\tiny\textcircled{\#}})=\R(A)$, $\N(A^{\tiny\textcircled{\#}})=\N(A^*)$.
The goal of Lemma \ref{LemTen12} is to verify that the tensor core inverse, introduced by the tensor equations $(C1)$, $(C2)$, $(3^T)$ in Definition \ref{CoreDefT}, satisfies $(1^T)$ and $(2^T)$.

\begin{lemma}\label{LemTen12}
Let $\mc{A}\in \mathbb{C}^{I(N) \times I(N)}$ be given.
Then $\mc{A}^{\core}$ satisfies $(1^T)$ and $(2^T)$.
\end{lemma}

\begin{proof}
The proof follows after the following verification:
\begin{eqnarray}
    \mc{A}\n\mc{A}^{\core}\n\mc{A}\!\!\!\!&=&\!\!\!\!\mc{A}\n\mc{A}^{\core}\n\mc{A}^{\core}\n\mc{A}^2=\mc{A}\n(\mc{A}^{\core})^2\n\mc{A}^2=\mc{A}^{\core}\n\mc{A}^2=\mc{A};\label{eq3.1}\\
    \mc{A}^{\core}\n\mc{A}\n\mc{A}^{\core}\!\!\!\!&=&\!\!\!\!\mc{A}^{\core}\n\mc{A}\n\mc{A}\n(\mc{A}^{\core})^2=\mc{A}^{\core}\n\mc{A}^2\n(\mc{A}^{\core})^2=\mc{A}\n(\mc{A}^{\core})^2=\mc{A}^{\core}.\label{eq3.2}
\end{eqnarray}
\end{proof}

The uniqueness of the core inverse is proved in Theorem \ref{ThmCoreUniq}.
\begin{theorem}\label{ThmCoreUniq}
An arbitrary core invertible tensor $\mc{A}\in \mathbb{C}^{I(N)\times I(N)}$ has one and only one core inverse.
\end{theorem}

\begin{proof}

Let $\mc{X}_1$ and $\mc{X}_2$ be two tensors satisfying $(C1)$, $(C2)$, $(3^T)$.
Using these properties in conjunction with (\ref{eq3.1}) and (\ref{eq3.2}), it follows that
\begin{equation*}
\aligned
\mc{A}\n\mc{X}_1&=\mc{A}\n\mc{X}_2\n\mc{A}\n\mc{X}_1=(\mc{A}\n\mc{X}_2)^*\n(\mc{A}\n\mc{X}_1)^*=\mc{X}_2^*\n\mc{A}^*\n\mc{X}_1^*\n\mc{A}^*\\
&=\mc{X}_2^*\n(\mc{A}\n\mc{X}_1\n\mc{A})^*=\mc{X}_2^*\n\mc{A}^*=(\mc{A}\n\mc{X}_2)^*=\mc{A}\n\mc{X}_2.
\endaligned
\end{equation*}
Thus
$$\mc{X}_1=\mc{X}_1\n\mc{A}\n\mc{X}_1=\mc{X}_1\n\mc{A}\n\mc{X}_2=\mc{X}_1\n\mc{A}^2\n\mc{X}_2^2=\mc{A}\n\mc{X}_2^2=\mc{X}_2.$$
Hence, the statement is proved.
\end{proof}

Corollary \ref{prep3.1} can be obtained by combining (\ref{eq3.1}), (\ref{eq3.2}) and $(3^T)$.

\begin{corollary}\label{prep3.1}
Let $\mc{A}\in \mathbb{C}^{I(N) \times I(N)}$ be a core tensor, then $\mc{A}^{\core}\in\mc{A}\{1^T,2^T,3^T\}.$
\end{corollary}

Theorem \ref{thm3.1} gives some characterizations of the core inverse in terms of other generalized inverses and generalizes Theorem 1 from \cite{baks}.
It is important to mention that the results of Theorem \ref{thm3.1}  are verified using one specific algebraic approach, which is new even in the matrix case.

\begin{theorem}\label{thm3.1}
Let $\mc{A}\in \mathbb{C}^{I(N) \times I(N)}$ be a core tensor. Then the following holds:
\begin{enumerate}
    \item[\bf (a)] $\mc{A}^{\core}=\mc{A}^{\#}\n\mc{A}\n\mc{A}^\dg=\mc{A}^{\#}\n{P}_{\mathcal{R}(\mathcal{A})}=\mc{A}\n\mc{A}^{\#}\n\mc{A}^\dg;$
    \item[\bf(b)] $\mc{A}^n\n \mc{A}^{\core}=\mc{A}^n\n \mc{A}^{\dg};$
    \item[\bf (c)] $\left(\mc{A}^{\core}\right)^\dg=\mc{A}^2\n\mc{A}^\dg=\mc{A}^2\n\mc{A}^{\core};$
    \item[\bf (d)] $\mc{A}^{\core}$ is EP;
    \item[\bf (e)] $\left(\mc{A}^{\core}\right)^{\core}=\mc{A}^2\n\mc{A}^\dg;$
    \item[\bf (f)] $\left(\mc{A}^{\core}\right)^2\n\mc{A}=\mc{A}^{\#};$
    \item[\bf(g)] $\mc{A}^{\core}\n\mc{A}=\mc{A}^{\#}\n\mc{A}$.
\end{enumerate}
\end{theorem}
\begin{proof}
\noindent {\bf (a)}
Since $\mc{A}$ is a core tensor, the existence of $\mc{A}^{\#}$ is ensured and later $\mc{X}:=\mc{A}^{\#}\n\mc{A}\n\mc{A}^\dg$.
It is necessary to verify $\mc{X}$ satisfies $(C1)$,$(C2)$,$(3^T)$.
Indeed:
$$
\aligned
& \mc{X}\n\mc{A}^2=\mc{A}^{\#}\n\mc{A}\n\mc{A}^\dg\n\mc{A}^2=\mc{A}^{\#}\n\mc{A}^2=\mc{A},\\
& \mc{A}\n\mc{X}^2=\mc{A}\n\mc{A}^{\#}\n\mc{A}\n\mc{A}^\dg\n\mc{A}^{\#}\n\mc{A}\n\mc{A}^\dg=\mc{A}\n\mc{A}^\dg\n\mc{A}\n\mc{A}^{\#}\n\mc{A}^\dg=\mc{A}^{\#}\n\mc{A}\n\mc{A}^\dg,\\
&(\mc{A}\n\mc{X})^*=(\mc{A}\n\mc{A}^\dg)^*=\mc{A}\n\mc{A}^\dg=\mc{A}\n\mc{A}^{\#}\n\mc{A}\n\mc{A}^\dg=\mc{A}\n\mc{X}.
\endaligned
$$

\noindent {\bf (b)}
The proof follows from the representations given in {\bf (a)} and simple transformations:
$$
\aligned
\mc{A}^n\n \mc{A}^{\core}=\mc{A}^{n-1}\n \mc{A}\n \mc{A}^{\#}\n\mc{A}\n\mc{A}^\dg=\mc{A}^n\n \mc{A}^{\dg}.
\endaligned
$$

\noindent {\bf (c)}
Let $\mc{Z}:=\mc{A}^2\n\mc{A}^\dg$. It is easy to show
$\mc{A}^{\core}\n\mc{Z}=\mc{Z}\n\mc{A}^{\core}=\mc{A}\n\mc{A}^\dg.$
Since
$$
\aligned
&\mc{A}^{\core}\n\mc{Z}\n\mc{A}^{\core}=\mc{A}\n\mc{A}^\dg\n\mc{A}^{\#}\n\mc{A}\n\mc{A}^\dg =\mc{A}^{\#}\n\mc{A}\n\mc{A}^\dg=\mc{A}^{\core},\\
&\mc{Z}\n\mc{A}^{\core}\n\mc{Z}=\mc{A}\n\mc{A}\n\mc{A}^\dg\n\mc{A}\n\mc{A}^\dg=\mc{Z},\\
&(\mc{A}^{\core}\n\mc{Z})^*=(\mc{Z}\n\mc{A}^{\core})^*=(\mc{A}\n\mc{A}^\dg)^*=(\mc{A}\n\mc{A}^\dg),
\endaligned
$$
the tensor $\mc{Z}$ is the Moore-Penrose inverse of $\mc{A}^{\core}.$

\noindent {\bf (d)} This part can be demonstrated after easy verification of the equalities
\[
A^{\core}\n(A^{\core})^\dg=(A^{\core})^\dg\n A^{\core}=\mc{A}\n\mc{A}^\dg.
\]

\noindent {\bf (e)} It is enough to show $A^{\core}\n(\mc{A}^2\n\mc{A}^\dg)^2=\mc{A}^2\n\mc{A}^\dg$ and $(\mc{A}^2\n\mc{A}^\dg)\n{\left(\mc{A}^{\core}\right)}^2=\mc{A}^{\core}$.
Indeed:
$$
\aligned
& \mc{A}^{\core}\n(\mc{A}^2\n\mc{A}^\dg)^2=\mc{A}\n\mc{A}^\dg\n\mc{A}^2\n\mc{A}^\dg=\mc{A}^2\n\mc{A}^\dg,\\
&\mc{A}^2\n\mc{A}^\dg\n \left(\mc{A}^{\core}\right)^2=\mc{A}\n\mc{A}^\dg\n\mc{A}^{\core}=\mc{A}\n\mc{A}^\dg\n\mc{A}^{\#}\n\mc{A}\n\mc{A}^\dg=\mc{A}\n\mc{A}^{\#}\n\mc{A}^\dg=\mc{A}^{\core}.
\endaligned
$$

\noindent {\bf (f)} This statement follows from $\left(\mc{A}^{\core}\right)^2\n\mc{A}=\mc{A}^{\#}\n\mc{A}\n\mc{A}^\dg\n\mc{A}^{\#}\n\mc{A}\n\mc{A}^\dg\n\mc{A}=\mc{A}^{\#}.$

\noindent {\bf (g)} The last one follows from {\bf (a)} and the definition of the Moore-Penrose inverse.
\end{proof}

\begin{remark}
It is easy to verify that $\mc{A}^{\core}=\mc{A}^{\#}\n\mc{A}\n\mc{A}^\dg=\mc{A}^{\#}\n{P}_{\mathcal{R}(\mathcal{A})}=\mc{A}\n\mc{A}^{\#}\n\mc{A}^\dg$ satisfies \eqref{CoreDefT}.
\end{remark}
\begin{corollary}
If $\mc{A}$ is a core tensor with full rank factorization $\mc{A}=\mc{P}\mc{Q}$, then
$$
\mc{A}^{\core}=\mc{A}^{(2)}_{\R(\mc{P}),\N(\mc{P^*})}.
$$
\end{corollary}
\begin{proof}
Using full rank factorizations of the group inverse and the Moore-Penrose inverse, it follows that
$$
\aligned
\mc{A}^{\core}&=\mc{A}\n\mc{A}^{\#}\n\mc{A}^\dg=\mc{P}\n\mc{Q}\n\mc{P}\n(\mc{Q}\n\mc{P})^{-2}\n \mc{Q}\n \mc{Q^*}\n(\mc{Q}\n\mc{Q^*})^{-1}\n(\mc{P^*}\n\mc{P})^{-1}\n\mc{P^*}\\
&= \mc{P}\n(\mc{P^*}\n\mc{A}\n \mc{P})^{-1}\n \mc{P^*},
\endaligned
$$
which clearly coincides with $\mc{A}^{(2)}_{\R(\mc{P}),\N(\mc{P^*})}$.
\end{proof}

The following representation of $\mc{A}^{\core}$ follows from Theorem \ref{thm3.1}(c).
\begin{corollary}
The core inverse can be obtained by $\mc{A}^{\core}=\left(\mc{A}^2\n\mc{A}^\dg\right)^\dg.$
\end{corollary}

We can also generalize Theorem \ref{thm3.1}(a) by using $\{1,3\}$ inverses of $\mc{A}$ instead of the Moore-Penrose inverse.
The result is proved in Theorem \ref{thm3.61}.

\begin{theorem}\label{thm3.61}
 Let $\mc{A}\in \mathbb{C}^{I(N) \times I(N)}$ be a core tensor.
 If there exists a tensor $\mc{X}:=\mc{A}^{(1^T,3^T)}\in \mathbb{C}^{I(N) \times I(N)}$ such that $\mc{X}\in\mc{A}\{1^T,3^T\},$ then $\mc{A}^{\core}=\mc{A}^{\#}\n\mc{A}\n\mc{X}.$
\end{theorem}
\begin{proof}
  Let us assume that there exists a tensor $\mc{X}:=\mc{A}^{(1^T,3^T)}$ satisfying $\mc{A}\n\mc{X}\n\mc{A}$ and $(\mc{A}\n\mc{X})^*=\mc{A}\n\mc{X}.$
Consider $\mc{Y}:=\mc{A}^{\#}\n\mc{A}\n\mc{A}^{(1^T,3^T)}.$
We can notice  the following expressions
$$
\aligned
& \mc{A}\n\mc{Y}^2=\mc{A}\n\mc{A}^{\#}\n\mc{A}\n\mc{A}^{(1^T,3^T)}\n \mc{A}^{\#}\n\mc{A}\n\mc{A}^{(1^T,3^T)}\\
&\qquad \quad \ = \mc{A}\n\mc{A}^{(1^T,3^T)}\n \mc{A}^{\#}\n\mc{A}\n\mc{A}^{(1^T,3^T)}=\mc{A}^{\#}\n\mc{A}\n\mc{A}^{(1^T,3^T)}=
\mc{Y},\\ &\mc{Y}\n\mc{A}^2=\mc{A}^{\#}\n\mc{A}\n\mc{A}^{(1^T,3^T)}\n\mc{A}^2=\mc{A}^{\#}\n\mc{A}^2=\mc{A},\\
& (\mc{A}\n\mc{Y})^*=(\mc{A}\n\mc{A}^{(1^T,3^T)})^*=\mc{A}\n\mc{A}^{(1^T,3^T)}=\mc{A}\n\mc{A}^{\#}\n\mc{A}\n\mc{A}^{(1^T,3^T)}=\mc{A}\n\mc{Y}.
\endaligned
$$
Therefore, $\mc{Y}$ satisfies $(C1)$, $(C2)$ and $(3^T)$ with respect to $\mc{A}.$
By the uniqueness of the core inverse, it follows that $\mc{A}^{\core}=\mc{A}^{\#}\n\mc{A}\n\mc{X}.$
\end{proof}

\begin{corollary}\label{CoreCoreEqu}
Let $\mc{A}\in \mathbb{C}^{I(N)\times I(N)}$ be a core tensor and $\mc{X}\in \mathbb{C}^{I(N)\times I(N)}.$
If $\mc{X}$ satisfies one of the following two systems {\bf (i)} or {\bf (ii)}
\begin{enumerate}
    \item[\bf (i)] $(C1)$\ \ $\mc{X}\n\mc{A}^2=\mc{A},$ and $(3^T)$\ \ $(\mc{A}\n\mc{X})^*=\mc{A}\n\mc{X};$
    \item[\bf (ii)]  $\mc{X}\n\mc{A}=\mc{A}^{\#}\n\mc{A},$ and \ $(3^T)$\ \ $(\mc{A}\n\mc{X})^*=\mc{A}\n\mc{X};$
\end{enumerate}
then $\mc{A}^{\core}=\mc{A}^{\#}\n\mc{A}\n\mc{X}.$
\end{corollary}

\begin{proof}
Let the system {\bf (i)} be satisfied.
If $\mc{X}\n\mc{A}^2=\mc{A}$ is true, then
$$\mc{X}\n\mc{A}=\mc{X}\n\mc{A}\n\mc{A}^{\#}\n\mc{A}=\mc{X}\n\mc{A}^2\n\mc{A}^{\#}=\mc{A}\n\mc{A}^{\#}=\mc{A}^{\#}\n\mc{A}.$$
Therefore,
$\mc{A}\n\mc{X}\n\mc{A}=\mc{A}$ and $\mc{X}\in \mc{A}\{1,3\}$.

Further if the conditions {\bf (ii)} are satisfied $\mc{X}\n\mc{A}=\mc{A}\n\mc{A}^{\#},$ then
$$\mc{A}\n\mc{X}\n\mc{A}=\mc{A}\n\mc{A}^{\#}\n\mc{A}=\mc{A},$$
which again implies $\mc{X}\in \mc{A}\{1,3\}$.

Now, the proof in both cases follows from Theorem \ref{thm3.61}.
\end{proof}

Corollary \ref{CorCorebc} follows from \cite[Theorem 4.3]{stan} and the properties of the core tensor inverse.

\begin{corollary}\label{CorCorebc}
Let ${\mathcal A}\in \mathbb C^{I(N) \times I(N)}$ is a given core tensor.

\smallskip
\noindent {\bf (a)} The following statements are equivalent:
\begin{itemize}\parskip0pt
\item[{\rm (i)}] there exists an outer inverse ${\mathcal X}\in \mathbb C^{I(N)\times I(N)}:=\mc{A}^{\core}={\mathcal A}_{\R({\mc{A}}),\N(\mc{A^*})}^{(2)}$;
\item[{\rm (ii)}] there exist a tensor ${\mathcal U}\in \mathbb C^{I(N) \times I(N)}$ that fulfils the tensor equations
\begin{equation}\label{EquCor1}
{\mc{A}}\n \mathcal{U}\n \mc{A^*}\n \mc{A}^2=\mc{A}, \ \ \mc{A^*}\n \mc{A}^2\n \mathcal{U}\n \mc{A^*}=\mc{A^*};
\end{equation}
\item[{\rm (iii)}] there exist tensors $\mathcal{U},\mathcal{V}\in \Bbb C^{I(N) \times I(N)}$ such that
\[
\mc{A}\n \mathcal{U}\n \mc{A^*}\n \mc{A}^2=\mc{A},\ \ \mc{A^*}\n \mc{A}^2\n \mathcal{V}\n \mc{A^*}=\mc{A^*};
\]
\item[{\rm (iv)}] there exist $\mathcal{U}\in \Bbb C^{I(N) \times I(N)}$ and $\mathcal{V}\in \Bbb C^{I(N) \times I(N)}$ such that
   \[
   \mc{A}\n \mathcal{U}\n \mc{A}^2=\mc{A},\ \ \mc{A^*}\n \mc{A}\n \mathcal{V}\n \mc{A^*}=\mc{A^*},\ \   \mc{A}\n \mathcal{U}=\mathcal{V}\n \mc{A^*};
   \]
\item[{\rm (v)}] there exist $\mathcal{U}\in \Bbb C^{I(N) \times I(N)}$ and $\mathcal{V}\in \Bbb C^{I(N) \times I(N)}$ satisfying
 \[
 \mc{A^*}\n \mc{A}^2\n \mathcal{U}=\mc{A^*},\ \ \mathcal{V}\n \mc{A^*}\n \mc{A}^2=\mc{A};
 \]
\item[{\rm (vi)}] $\N(\mc{A^*}\n \mc{A}^2)=\N(\mc{A})$;
\item[{\rm (vii)}] $\R(\mc{A^*}\n \mc{A}^2)=\R(\mc{A^*})$;
\item[{\rm (viii)}] $\mc{A}\n (\mc{A^*}\n \mc{A}^2)^{(1)}\n \mc{A^*}\n \mc{A}^2=\mc{A}$ and
     $\mc{A^*}\n \mc{A}^2\n (\mc{A^*}\n \mc{A}^2)^{(1)}\n \mc{A^*}=\mc{A^*}$;
\item[{\rm (viii)}] $\rra{\mc{A^*}\n \mc{A}^2}=\rra{\mc{A}}$;
\item[{\rm (ix)}] $A$ is core invertible;
\end{itemize}

\noindent {\bf (b)} If the statements in {\bf (a)} are true, then
\begin{equation}\label{eq:2inv.pRN}
\aligned
\mc{A}^{\core}&=\mc{A}\n (\mc{A^*}\n \mc{A}^2)^{(1)}\n \mc{A^*}\\
&=\mc{A}\n \mathcal{U}\n \mc{A^*},\ \
\endaligned
\end{equation}
for arbitrary 
$\mathcal{U}\in \Bbb C^{K(k)\times L(l)}$ satisfying \eqref{EquCor1}.
\end{corollary}

The next result gives a correlation between the core inverse of a tensor power and the power of its core inverse.
\begin{theorem}\label{thm3.2}
Let $\mc{A}\in \mathbb{C}^{I(N) \times I(N)}$ be a core tensor.
Then the following holds for an arbitrary integer $m\geq 1$:
\begin{enumerate}
    \item[\bf (a)]  $\mc{A}^m\n\left(\mc{A}^{\core}\right)^m=\mc{A}\n\mc{A}^\dg,$
    \item[\bf (b)] $\left(\mc{A}^{\core}\right)^m =\left(\mc{A}^m\right)^{\core}.$
\end{enumerate}
\end{theorem}

\begin{proof}
{\bf (a)} Using the statement of Theorem \ref{thm3.1}(a), we obtain
\begin{equation*}
\aligned
&\mc{A}^m\n\left(\mc{A}^{\core}\right)^m=\mc{A}^m\n\left(\mc{A}^{\#}\n\mc{A}\n\mc{A}^\dg\right)^m\\
&~~~~=\mc{A}^m\n\mc{A}^\dg\n\left(\mc{A}^{\#}\n\mc{A}\n\mc{A}^\dg\right)^{m-1}= \mc{A}^{m}\n\mc{A}^{\#}\n\mc{A}^\dg\n\left(\mc{A}^{\#}\n\mc{A}\n\mc{A}^\dg\right)^{m-2}\\
&~~~~= \mc{A}^{m-1}\n\mc{A}^\dg\n\left(\mc{A}^{\#}\n\mc{A}\n\mc{A}^\dg\right)^{m-2}= \mc{A}^{m-2}\n\mc{A}^\dg\n\left(\mc{A}^{\#}\n\mc{A}\n\mc{A}^\dg\right)^{m-3}.
\endaligned
\end{equation*}
Continuing in the same way, one can obtain
$$
\aligned
\mc{A}^m\n\left(\mc{A}^{\core}\right)^m= \mc{A}^2\n\mc{A}^\dg\n\left(\mc{A}^{\#}\n\mc{A}\n\mc{A}^\dg\right)=\mc{A}\n\mc{A}^\dg.
\endaligned
$$
{\bf (b)} Let $\mc{X}=\left(\mc{A}^{\core}\right)^m.$
Now, using the part {\bf (a)} of this theorem, we get
$$
\aligned
&\mc{X}\n(\mc{A}^m)^2=\left(\mc{A}^{\#}\n\mc{A}\n\mc{A}^\dg\right)^m\n\mc{A}^{2m}=\left(\mc{A}^{\#}\n\mc{A}\n\mc{A}^\dg\right)^{m-1}\n\mc{A}^{2m-1}=\cdots =\mc{A}^m,\\ &\mc{A}^m\n\mc{X}^2=\mc{A}\n\mc{A}^\dg\n\mc{X}=\mc{A}\n\mc{A}^\dg\n\left(\mc{A}^{\#}\n\mc{A}\n\mc{A}^\dg\right)^{m}=\left(\mc{A}^{\#}\n\mc{A}\n\mc{A}^\dg\right)^m=\mc{X},\\ &\left(\mc{A}^m\n\mc{X}\right)^*=\left(\mc{A}\n\mc{A}^\dg\right)^*=\mc{A}\n\mc{A}^\dg=\mc{A}^m\n\mc{X}.
\endaligned
$$
Therefore, $\mc{X}$ is the core inverse of $\mc{A}^m$, which completes the proof.
\end{proof}

In general, the Moore-Penrose inverse, the group inverse and the core inverse of tensors are different.
This fact is confirmed in Example \ref{Example1}.

\begin{example}\label{Example1}
Consider a tensor
$\mc{A}=(\mc{A}_{ijkl})
 \in \mathbb{R}^{(2\times3)\times (2\times3)}$ with entries
\begin{eqnarray*}
\mc{A}_{ij11} =
    \begin{bmatrix}
    1 & 6 & 2 \\
    -1 & 3 &  0
    \end{bmatrix},~
\mc{A}_{ij12} = \mc{A}_{ij13} =\mc{A}_{ij21} =  \mc{A}_{ij22} = a_{ij23} =
    \begin{bmatrix}
     0 & 0 & 0\\
     0 & 0 & 0
    \end{bmatrix},
\end{eqnarray*}
Then, $\mc{A}^\# = \mc{A} \in \mathbb{R}^{(2\times3)\times (2\times3)}$.
However, $ \mc{A}^\dg = (\mc{X}_{ijkl}) \in \mathbb{R}^{(2\times3)\times (2\times3)}$, where
\begin{eqnarray*}
\mc{X}_{ij11} =
    \begin{bmatrix}
    1/51 & 0 & 0 \\
    0 & 0 & 0
    \end{bmatrix},~
\mc{X}_{ij12} =
    \begin{bmatrix}
  2/17 & 0 & 0 \\
    0 & 0 & 0
    \end{bmatrix},~
\mc{X}_{ij13} =
    \begin{bmatrix}
  2/51 & 0 & 0 \\
    0 & 0 & 0
\end{bmatrix},\\
\mc{X}_{ij21} =
    \begin{bmatrix}
  -1/51 & 0 & 0 \\
    0 & 0 & 0
    \end{bmatrix},~
 \mc{X}_{ij22} =
    \begin{bmatrix}
  1/17 & 0 & 0 \\
    0 & 0 & 0
    \end{bmatrix},~
\mc{X}_{ij23} =
    \begin{bmatrix}
     0 & 0 & 0\\
     0 & 0 & 0
    \end{bmatrix}.
\end{eqnarray*}
and $ \mc{A}^{\core} = (\mc{Y}_{ijkl}) \in \mathbb{R}^{(2\times3)\times (2\times3)}$, where
\small{
\begin{eqnarray*}
&&\mc{Y}_{ij11} =
    \begin{bmatrix}
    1/51 & 2/17 & 2/51 \\
    -1/51 & 1/17 & 0
    \end{bmatrix},~
\mc{Y}_{ij12} =
    \begin{bmatrix}
  2/17 & 12/17 & 4/17 \\
    -2/17 & 6/17 & 0
    \end{bmatrix},~
\mc{Y}_{ij13} =
    \begin{bmatrix}
  2/51 & 4/17 & 4/51 \\
    -2/51 & 2/17 & 0
\end{bmatrix},\\
&&\mc{Y}_{ij21} =
    \begin{bmatrix}
  -1/51 & 0 & 0 \\
    0 & 0 & 0
    \end{bmatrix},~
\mc{Y}_{ij22} =
    \begin{bmatrix}
  1/17 & 6/17 & 2/17 \\
    -1/17 & 3/17 & 0
    \end{bmatrix},~
\mc{Y}_{ij23} =
    \begin{bmatrix}
     0 & 0 & 0\\
     0 & 0 & 0
    \end{bmatrix}.
\end{eqnarray*}
}
Hence it is clear that $\mc{A}^{\core}, \mc{A}^\#$ and $\mc{A}^\dg$ are different.
\end{example}

Now, a natural question is: are these inverses the same under some assumptions?
The next theorem will provide an affirmative  answer to this along with some characterizations for EP tensors.
The results of Theorem \ref{thm3.3} generalize the results of Theorem 2 from \cite{baks}.

\begin{theorem}\label{thm3.3}
 Let $\mc{A}\in \mathbb{C}^{I(N) \times I(N)}$ be a core tensor.  Then the following statements hold:
\begin{enumerate}
    \item[\bf (a)] $\mc{A}^{\core}=0$ if and only if $\mc{A}=0,$
    \item[\bf (b)] $\mc{A}^{\core}=\mc{A}\n\mc{A}^{\dg}$ if and only if $\mc{A}$ idempotent,
    \item[\bf (c)] $\mc{A}^{\core}=\mc{A}^{\dg}$ if and only if $\mc{A}$ is EP,
    \item[\bf (d)] $\mc{A}^{\core}=\mc{A}^{\#}$ if and only if $\mc{A}$ is EP,
     \item[\bf (e)] $\mc{A}^{\core}=\mc{A}$ if and only if $\mc{A}$ is EP and tripotent,
    \item[\bf (f)] $\mc{A}^{\core}=\mc{A}^{*}$ if and only if $\mc{A}$ is EP and partial isometry.
\end{enumerate}
\end{theorem}
\begin{proof}
{\bf (a)} If $\mc{A}=0,$ then $\mc{A}^{\core}=0$ follows trivially from Theorem \ref{thm3.1}(a).
To show the converse, let $\mc{A}^{\core}=0.$
This assumption leads to $\mc{A}^{\#}\n\mc{A}\n\mc{A}^{\dg}=0.$
Post-multiplying and pre-multiplying by $\mc{A}$, we get $\mc{A}=0.$
Thus, the statement in {\bf (a)} is completed.

\noindent {\bf (b)} Now let $\mc{A}^2=\mc{A}.$
Then $\mc{A}^{\core}=\mc{A}^{\#}\n\mc{A}\n\mc{A}^{\dg}=\mc{A}^{\#}\n\mc{A}^2\n\mc{A}^{\dg}=\mc{A}\n\mc{A}^{\dg}.$
The opposite statement is true since
$$\mc{A}=\mc{A}\n\mc{A}^{\#}\n\mc{A}=\mc{A}\n\mc{A}^{\#}\n\mc{A}\n\mc{A}^{\dg}\n\mc{A}=\mc{A}\n\mc{A}\n\mc{A}^{\dg}\n\mc{A}=\mc{A}^2.$$
\noindent {\bf (c)} Let us assume $\mc{A}^{\core}=\mc{A}^{\dg}.$
So, by Theorem \ref{thm3.1}(e), it follows that $\mc{A}^{\core}$ is EP.
This implies that $\mc{A}^{\core}\n\left(\mc{A}^{\core}\right)^{\dg}=\left(\mc{A}^{\core}\right)^{\dg}\n\mc{A}^{\core}.$
This yields $\mc{A}^{\dg}\n\mc{A}=\mc{A}\n\mc{A}^{\dg}$ since $\mc{A}^{\core}=\mc{A}^{\dg}.$
Therefore $\mc{A}$ is EP.

Now the converse follows by the following expression $$\mc{A}^{\core}=\mc{A}^{\#}\n\mc{A}\n\mc{A}^{\dg}=\mc{A}^{\#}\n\mc{A}\n\mc{A}^{\dg}\n\mc{A}\n\mc{A}^{\dg}=\mc{A}\n\mc{A}^{\#}\n\mc{A}^{\dg}\n\mc{A}\n\mc{A}^{\dg}=\mc{A}\n\mc{A}^{\dg}\n\mc{A}^{\dg}=\mc{A}^{\dg}.$$
\noindent {\bf (d)} Again, assume that $\mc{A}$ is EP.
Then the core inverse of $\mc{A}$ is equal to $$\mc{A}^{\core}=\mc{A}^{\#}\n\mc{A}\n\mc{A}^{\dg}=\mc{A}^{\#}\n\mc{A}\n\mc{A}^{\#}\n\mc{A}\n\mc{A}^{\dg}=\left(\mc{A}^{\#}\right)^2\n\mc{A}^2\n\mc{A}^{\dg}=\left(\mc{A}^{\#}\right)^2\n\mc{A}=\mc{A}^{\#}.$$
Consider $\mc{A}^{\core}=\mc{A}^{\#}$.
To claim that $\mc{A}$ is EP, it is enough to show $\mc{A}^{\#}=\mc{A}^{\dg}.$
Since $\mc{A}^{\core}$ is  $\{1^T,2^T,3^T\}$ inverse of $\mc{A}$, we only need to claim that $\left(\mc{A}^{\#}\n\mc{A}\right)^*=\mc{A}^{\#}\n\mc{A}$.
Since $\mc{A}^{\#}\n\mc{A}=\mc{A}\n\mc{A}^{\#},$
it implies
$\mc{A}\n\mc{A}^{\core}=\mc{A}^{\core}\n\mc{A}.$
Now
$$\left(\mc{A}^{\#}\n\mc{A}\right)^*=\left(\mc{A}^{\core}\n\mc{A}\right)^*=\left(\mc{A}\n\mc{A}^{\core}\right)^*=\mc{A}\n\mc{A}^{\core}=\mc{A}\n\mc{A}^{\#}=\mc{A}^{\#}\n\mc{A}.$$
This completes the proof for {\bf (d)}.

\noindent {\bf (e)} Next assume that $\mc{A}^{\core}=\mc{A}.$ Since $\mc{A}^{\core}$ is EP which yields $\mc{A}$ is EP.
From Definition \ref{CoreDefT}, it follows that $\mc{A}\n\left(\mc{A}^{\core}\right)^2=\mc{A}^{\core}.$
This turns to $\mc{A}^3=\mc{A}.$ So $\mc{A}$ is EP and tripotent.
The converse statement holds since
$$\mc{A}^{\core}=\mc{A}^{\#}\n\mc{A}\n\mc{A}^{\dg}=\mc{A}^{\#}\n\mc{A}^3\n\mc{A}^{\dg}=\mc{A}^{\#}\n\mc{A}^2\n\mc{A}^{\dg}\mc{A}=\mc{A}^{\#}\n\mc{A}^2=\mc{A}.$$
\noindent {\bf (f)} To show the last part, first consider $\mc{A}^{\core}=\mc{A}^*.$ Since $\mc{A}\n\mc{A}^{\core}\n\mc{A}=\mc{A},$ it gives $\mc{A}\n\mc{A}^{*}\n\mc{A}=\mc{A}.$
This gives $\mc{A}^{\core}=\mc{A}^*=\mc{A}^\dg.$
Then, by part {\bf (c)}, $\mc{A}$ is an isometry and EP.
The converse statement follows from part {\bf (c)} and the definition of isometry.
\end{proof}

Combining the results of Theorem \ref{thm3.3} (c) and (d), we state the following remark.
\begin{remark}\label{rem3.1}
Let $\mc{A}\in \mathbb{C}^{I(N) \times I(N)}$ be a core tensor.
Then  $\mc{A}$ is EP if and only if $\mc{A}^{\core}=\mc{A}^\dg=\mc{A}^{\#}.$
\end{remark}

Further characterizations of EP tensors are provided in the next theorem.
The results of Theorem \ref{thm3.4} generalize the results of Theorem 2 from \cite{baks}.

\begin{theorem}\label{thm3.4}
Let $\mc{A}\in \mathbb{C}^{I(N) \times I(N)}$ be a core tensor. Then the following are equivalent:
\begin{enumerate}
    \item[\rm (i)] $\mc{A}$ is EP,
    \item[\rm  (ii)] $\left(\mc{A}^{\core}\right)^{\core}=\mc{A},$
    \item[\rm  (iii)] $\mc{A}^{\core}\n\mc{A}=\mc{A}\n \mc{A}^{\core},$
    \item[\rm  (iv)] $\left(\mc{A}^{\dg}\right)^{\core}=\mc{A},$
    \item[\rm  (v)] $\left(\mc{A}^{\core}\right)^{\dg}=\mc{A}.$
\end{enumerate}
\end{theorem}
\begin{proof}
Let $\mc{A}$ is EP. By Theorem \ref{thm3.1}(d), $\left(\mc{A}^{\core}\right)^{\core}=\mc{A}^2\n\mc{A}^\dg=\mc{A}\n\mc{A}^\dg\n\mc{A}=\mc{A}.$
Thus $(i)\Rightarrow ~(ii).$

Assume $(ii)$ is true, i.e., $\mc{A}^2\n\mc{A}^\dg=\mc{A}.$
Now
$$\mc{A}^{\core}\n\mc{A}=\mc{A}^{\#}\n\mc{A}=\mc{A}^{\#}\n\mc{A}^2\n\mc{A}^\dg=\mc{A}\n\mc{A}^{\#}\n\mc{A}\n\mc{A}^\dg=\mc{A}\n\mc{A}^{\core}.$$
This completes $(ii)\Rightarrow ~(iii).$

Next we will claim $(iii)\Rightarrow ~(i).$
Consider $\mc{A}^{\core}\n\mc{A}=\mc{A}\n \mc{A}^{\core}.$
This turnout to $\mc{A}\n\mc{A}^{\#}=\mc{A}\n\mc{A}^\dg.$ Now $\mc{A}^{\core}=\mc{A}^{\#}\n\mc{A}\n\mc{A}^\dg=\mc{A}^{\#}\n\mc{A}\n\mc{A}^{\#}=\mc{A}^{\#}.$
So by Theorem \ref{thm3.3}(d), the tensor $\mc{A}$ is EP.

\smallskip
Next we will claim the equivalence between $(i)$ and $(iv)$.
Let us assume $\mc{A}$ is EP.
Then $\mc{A}^{\core}=\mc{A}^\dg$ by  Theorem \ref{thm3.3}(c).
This yields $\left(\mc{A}^\dg\right)^{\core}=\left(\mc{A}^{\core}\right)^{\core}=\mc{A}.$
To show the converse it is enough to show that $\mc{A}^\dg$ is the core inverse of $\mc{A}.$
Let $\left(\mc{A}^\dg\right)^{\core}=\mc{A}.$  Which implies $\mc{A}^\dg\n\mc{A}^2=\mc{A}$ and $\mc{A}\n\left(\mc{A}^\dg\right)^2=\mc{A}^\dg.$
As $\left(\mc{A}\n\mc{A}^\dg\right)^*=\mc{A}\n\mc{A}^\dg,$ it follows that $\mc{A}^\dg$ satisfies all the conditions of the core inverse.
Therefore $(i)\Leftrightarrow(iv)$.

\smallskip
Finally, we will show $(i)\Leftrightarrow (v).$
The tensor $\mc{A}$ is EP if and only if $\mc{A}^{\core}=\mc{A}^\dg.$
This turnout $\mc{A}$ is EP if and only if $\left(\mc{A}^{\core}\right)^\dg=(\mc{A}^\dg)^\dg=\mc{A}.$
\end{proof}
By looking at Theorem \ref{thm3.4} $(i),~(iv)$ and $(v),$ we state the following remark.
\begin{remark}\label{rm3.2}
Let $\mc{A}\in \mathbb{C}^{I(N) \times I(N)}$ be a core tensor. Then $\mc{A}$ is EP if and only if $\left(\mc{A}^{\core}\right)^{\dg}=\left(\mc{A}^{\dg}\right)^{\core}.$
\end{remark}

We will discuss some equivalent characterization of core inverse in the remaining part of this section. The condition $\mc{A}\n\mc{X}^2=\mc{X}$ can be dropped from the definition of core inverse but with an additional condition as stated in the theorem given below.

\begin{theorem}\label{thm3.14}
Let $\mc{A}\in \mathbb{C}^{I(N) \times I(N)}$ be a core tensor.
If there exists a tensor $\mc{X}$ satisfying \\
$(2^T)$\ $\mc{X}\n\mc{A}\n\mc{X}=\mc{X}$,\ $(3^T)$\  $(\mc{A}\n\mc{X})^*=\mc{A}\n\mc{X}$ and $(C1)$\ $\mc{X}\n\mc{A}^2=\mc{A},$\\
then $\mc{X}$ is the core inverse of $\mc{A}.$
\end{theorem}
\begin{proof}
 Let $\mc{A}=\mc{X}\n\mc{A}^2.$ Post-multiplying by $\mc{A}^{\#},$ we obtain $\mc{A}\n\mc{A}^{\#}=\mc{X}\n\mc{A}^2\n\mc{A}^{\#}=\mc{X}\n\mc{A}.$
 Now $\mc{A}\n\mc{X}\n\mc{A}=\mc{A}\n\mc{A}\n\mc{A}^{\#}=\mc{A}.$
 Therefore, $\mc{X}\in\mc{A}\{1^T,3^T\}.$
 So, by Theorem \ref{thm3.61}, the core inverse of $\mc{A}$ is $\mc{A}^{\core}=\mc{A}^{\#}\n\mc{A}\n\mc{X}=\mc{A}\n\mc{A}^{\#}\n\mc{X}.$
 Replacing $\mc{A}\n\mc{A}^{\#}$ by $\mc{X}\n\mc{A},$ one can obtain $\mc{A}^{\core}=\mc{X}\n\mc{A}\n\mc{X}=\mc{X}.$
  Hence the proof is complete.
\end{proof}
Theorem \ref{thm3.14} leads to the following corollary.
\begin{corollary}\label{cor3.15}
Let $\mc{A}\in \mathbb{C}^{I(N) \times I(N)}$ be a core tensor.
If there exists a tensor $\mc{X}$ satisfying \\
$(2^T)$\  $\mc{X}\n\mc{A}\n\mc{X}=\mc{X},$\ $(3^T)$\  $(\mc{A}\n\mc{X})^*=\mc{A}\n\mc{X}$ and $(C1)$\ $\mc{X}\n\mc{A}^2=\mc{A},$ \\
then $\mc{A}\n\mc{X}^2=\mc{X}.$
\end{corollary}

\begin{remark}
Corollary \ref{cor3.15} can be proved separately and the proof is the following.
Since $\mc{X}=\mc{X}\n\mc{A}\n\mc{X}=\mc{X}^2\n\mc{A}
\n\mc{X}=\mc{X}^2\n\mc{A}^2\n\mc{X}.$ Pre-multiplying by $\mc{X}^{\#},$ we obtain
\[
\mc{X}^{\#}\n\mc{X}=\mc{X}^{\#}\n\mc{X}^2\n\mc{A}^2\n\mc{X}=\mc{X}\n\mc{A}^2\n\mc{X}=\mc{A}\n\mc{X}.
\]
Now $\mc{A}\n\mc{X}^2=\mc{A}\n\mc{X}\n\mc{X}=\mc{X}^{\#}\n\mc{X}\n\mc{X}=\mc{X}.$
The assumptions of Corollary \ref{cor3.15}  can also be used as an equivalent definition of core inverse of $\mc{A}.$
\end{remark}

Next, we discuss another equivalent definition of the core inverse, which was proved in \cite{wang2015} for the matrix case by using the singular decomposition.

\begin{proposition}\label{prep3.16}
Let $\mc{A}\in \mathbb{C}^{I(N) \times I(N)}$ be a core tensor.
If there exists a tensor $\mc{X}$ satisfying\\
$(1^T)$\  $\mc{A}\n\mc{X}\n\mc{A}=\mc{A}$,\ $(3^T)$\  $(\mc{A}\n\mc{X})^*=\mc{A}\n\mc{X}$ and $(C2)$\ $\mc{A}\n\mc{X}^2=\mc{X},$ \\
then $\mc{X}$ is core inverse of $\mc{A}.$
\end{proposition}
\begin{proof}
It is sufficient to show that $\mc{X}\n\mc{A}^2=\mc{A}.$ Combining the result $\mc{A}\n\mc{X}\n\mc{A}=\mc{A}$ and $\mc{A}\n\mc{X}^2=\mc{X}$, we get  $\mc{A}=\mc{A}^2\n\mc{X}^2\n\mc{A}.$
Pre-multiplying $\mc{A}^{\#}$ both sides, we obtain $\mc{A}^{\#}\n\mc{A}=\mc{A}^{\#}\mc{A}^2\n\mc{X}^2\n\mc{A}=\mc{A}\n\mc{X}^2\n\mc{A}=\mc{X}\n\mc{A}.$
Now $\mc{X}\n\mc{A}^2=\mc{X}\n\mc{A}\n\mc{A}=\mc{A}^{\#}\n\mc{A}\n\mc{A}=\mc{A}.$ Hence $\mc{X}$ is the core inverse of $\mc{A}.$
\end{proof}

By combining Theorem \ref{thm3.14} and Proposition \ref{prep3.16}, we state the following corollary which can be used as an equivalent characterization of the core inverse.
\begin{corollary}\label{eqvcore}
Let $\mc{A}\in \mathbb{C}^{I(N) \times I(N)}$ be a core tensor and $\mc{X}\in \mathbb{C}^{I(N) \times I(N)}$.
If $\mc{X}$ satisfies any one of the following triads of equations:
\begin{enumerate}
    \item[\bf (i)]\quad $(C1)$\ $\mc{X}\n\mc{A}^2=\mc{A},$\quad $(C2)$\  $\mc{A}\n\mc{X}^2=\mc{X},$  and $(3^T)$\ $(\mc{A}\n\mc{X})^*=\mc{A}\n\mc{X}$,
    \item[\bf (ii)]\quad $(2^T)$\ $\mc{X}\n\mc{A}\n\mc{X}=\mc{X},$\quad $(C1)$\  $\mc{X}\n\mc{A}^2=\mc{A},$ and $(3^T)$\ $(\mc{A}\n\mc{X})^*=\mc{A}\n\mc{X}$,
    \item[\bf (iii)]\quad $(1^T)$\ $\mc{A}\n\mc{X}\n\mc{A}=\mc{A},$\quad $(2^T)$\  $\mc{A}\n\mc{X}^2=\mc{X},$ and $(3^T)$\ $(\mc{A}\n\mc{X})^*=\mc{A}\n\mc{X}$,
\end{enumerate}
then $\mc{X}$ is the core inverse of $\mc{A}.$
\end{corollary}

If we replace the condition $\mc{A}\n\mc{X}^2=\mc{X}$ in Corollary \ref{eqvcore} {\bf (iii)} by $\mc{X}\n\mc{A}^2=\mc{A},$ then one can easily verify that $\mc{X}\n\mc{A}\n\mc{X}$ is the core inverse of $\mc{A}.$
In conclusion, we pose the following remark.
\begin{remark}
Let $\mc{A}\in \mathbb{C}^{I(N) \times I(N)}$ be a core tensor.
 If there exists a tensor $\mc{X}$ satisfying $\mc{A}\n\mc{X}\n\mc{A}=\mc{A},$ $(\mc{A}\n\mc{X})^*=\mc{A}\n\mc{X},$ and $\mc{X}\n\mc{A}^2=\mc{A},$
then $\mc{X}\n\mc{A}\n\mc{X}$ is the core inverse of $\mc{A}.$
\end{remark}

The core inverse can be further characterized using the range condition as in Theorem \ref{rangequiv}.
\begin{theorem}\label{rangequiv}
Let $\mc{A}\in \mathbb{C}^{I(N) \times I(N)}$ be a core tensor and  $\mc{X}\in \mathbb{C}^{I(N) \times I(N)}.$
If $\mc{X}$ satisfies  $(\mc{A}\n\mc{X})^*=\mc{A}\n\mc{X},$  $\mc{X}\n\mc{A}^2=\mc{A},$ and $\mathfrak{R}(\mc{X})\subseteq \mathfrak{R}(\mc{A}),$
then $\mc{X}$ is the core inverse of $\mc{A}$.
\end{theorem}
\begin{proof}
  To show $\mc{X}$ is the core inverse of $\mc{A},$  it is enough to show only $\mc{X}\n\mc{A}\n\mc{X}=\mc{X}.$ Since $\mathfrak{R}(\mc{X})\subseteq\mathfrak{R}(\mc{A}).$ So by Lemma \ref{range-stan}, there exists a tensor $\mc{Z}\in \mathbb{C}^{I_1\times\cdots\times
I_N \times I_1 \times\cdots\times I_N }$ such that $\mc{X}=\mc{A}\n\mc{Z}.$ From the condition  $\mc{X}\n\mc{A}^2=\mc{A},$ we obtain $$
\mc{X}\n\mc{A}\n\mc{X}=\mc{X}\n\mc{A}\n\mc{A}\n\mc{Z}=\mc{X}\n\mc{A}^2\n\mc{Z}=\mc{A}\n\mc{Z}=\mc{X}.
$$
Therefore, by Theorem \ref{thm3.14}, $\mc{X}$ is the core inverse of $\mc{A}.$
\end{proof}
Note that the converse of Theorem \ref{rangequiv}, is also true and can be verified from the definition of the core inverse.
In the next subsection, we discuss the computation of core inverse for sum of two tensors.

\subsection{Tensor core inverse of the sum of two tensors}\label{SubsecCoreSUm}

Group inverse of sum of two elements over a ring was first introduced by Drazin \cite{draz}.
Some generalization also discussed in \cite{chen2009}. Then the same idea extended to core inverse on a ring in \cite{xu2017}.
The result of \cite{xu2017}, further extended to tensors in \cite{DrazBehAsJ} and stated below.
\begin{theorem}\label{gpsum}
 Let $\mc{A},~ \mc{B}\in \mathbb{C}^{I_1\times\cdots\times
I_N \times I_1 \times\cdots\times I_N }$ be core tensors with $\mc{A}\n\mc{B}=\mc{O}.$ Then
$$
\left(\mc{A}+\mc{B}\right)^{\#}=\left(\mc{I}-\mc{B}\n\mc{B}^{\#}\right)\n\mc{A}^{\#}+\mc{B}^{\#}\n\left(\mc{I}-\mc{A}\n\mc{A}^{\#}\right).
$$
\end{theorem}

In the next theorem, we generalize the same result with one more additional condition.
\begin{theorem}\label{crsum}
 Let $\mc{A},~ \mc{B}\in \mathbb{C}^{I_1\times\cdots\times
I_N \times I_1 \times\cdots\times I_N }$ be core tensors with $\mc{A}\n\mc{B}=\mc{O}$ and $\mc{A}^*\n\mc{B}=\mc{O}.$ Then
$
\left(\mc{A}+\mc{B}\right)^{\core}=\left(\mc{I}-\mc{B}\n\mc{B}^{\core}\right)\n\mc{A}^{\core}+\mc{B}^{\core}.$
\end{theorem}
\begin{proof}
By Theorem \ref{gpsum}, it follows that
\[
 \left(\mc{A}+\mc{B}\right)^{\#}=\left(\mc{I}-\mc{B}\n\mc{B}^{\#}\right)\n\mc{A}^{\#}+\mc{B}^{\#}\n\left(\mc{I}-\mc{A}\n\mc{A}^{\#}\right)
 =\left(\mc{I}-\mc{B}^{\#}\n\mc{B}\right)\n\mc{A}^{\#}+\mc{B}^{\#}\n\left(\mc{I}-\mc{A}^{\#}\n\mc{A}\right).
 \]
Using Theorem \ref{thm3.1}(f), we further obtain
 \begin{eqnarray*}
 \left(\mc{A}+\mc{B}\right)^{\#}&=&\left(\mc{I}-(\mc{B}^{\core})^2\n\mc{B}^2\right)\n(\mc{A}^{\core})^2\n\mc{A}+(\mc{B}^{\core})^2\n\mc{B}\n\left(\mc{I}-(\mc{A}^{\core})^2\n\mc{A}^2\right)\\
 &=&\left(\mc{I}-\mc{B}^{\core}\n\mc{B}\right)\n(\mc{A}^{\core})^2\n\mc{A}+(\mc{B}^{\core})^2\n\mc{B}\n\left(\mc{I}-\mc{A}^{\core}\n\mc{A}\right).
 \end{eqnarray*}
 Since $\mc{A}\n\mc{B}=\mc{O}$ and $\mc{A}^*\n\mc{B}=\mc{O}$, the following transformations follows:
 $$
 \aligned
 & \mc{A}\n\mc{B}^{\core}=\mc{A}\n\mc{B}\n(\mc{B}^{\core})^2=\mc{O},\\
 &\mc{B}^{\core}\n\mc{A}=\mc{B}^{\core}\n\mc{B}\n\mc{B}^{\core}\n\mc{A}=\mc{B}^{\core}\n(\mc{B}\n\mc{B}^{\core})\n\mc{A}\\
 & \qquad \quad \ \ =\mc{B}^{\core}\n(\mc{B}^{\core})^*\n\mc{B}^{*}\n\mc{A}=\mc{B}^{\core}\n(\mc{B}^{\core})^*\n(\mc{A}^*\n\mc{B})^*=\mc{O},\\
 & \mc{A}^{\core}\n\mc{B}=\mc{A}^{\core}\n\mc{A}\n\mc{A}^{\core}\n\mc{B}=\mc{A}^{\core}\n(\mc{A}\n\mc{A}^{\core})^*\n\mc{B}=\mc{A}^{\core}\n(\mc{A}^{\core})^*\n\mc{A}^*\n\mc{B}=\mc{O}.
 \endaligned
 $$
 Let $\mc{X}=\left(\mc{I}-\mc{B}\n\mc{B}^{\core}\right)\n\mc{A}^{\core}+\mc{B}^{\core}.$
 Now we have
 \begin{equation*}
 \aligned
 (\mc{A}+\mc{B})\n\mc{X}&=(\mc{A}+\mc{B})\n\left[\left(\mc{I}-\mc{B}\n\mc{B}^{\core}\right)\n\mc{A}^{\core}+\mc{B}^{\core}\right]\\
 &=\mc{A}\n\left(\mc{I}-\mc{B}\n\mc{B}^{\core}\right)\n\mc{A}^{\core}+\mc{B}\n\left(\mc{I}-\mc{B}\n\mc{B}^{\core}\right)\n\mc{A}^{\core}+\mc{A}\n\mc{B}^{\core}+\mc{B}\n\mc{B}^{\core}\\
  &=\mc{A}\n\left(\mc{I}-\mc{B}\n\mc{B}^{\core}\right)\n\mc{A}^{\core}+\mc{B}\n\mc{B}^{\core}=\mc{A}\n\mc{A}^{\core}+\mc{B}\n\mc{B}^{\core}.
  \endaligned
 \end{equation*}
 Thus $\left((\mc{A}+\mc{B})\n\mc{X}\right)^*=(\mc{A}\n\mc{A}^{\core})^*+(\mc{B}\n\mc{B}^{\core})^*=(\mc{A}+\mc{B})\n\mc{X}.$ Also we have
  \begin{eqnarray*}
 (\mc{A}+\mc{B})\n\mc{X}\n(\mc{A}+\mc{B})&=&(\mc{A}\n\mc{A}^{\core}+\mc{B}\n\mc{B}^{\core})\n(\mc{A}+\mc{B})\\
 &=&\mc{A}\n\mc{A}^{\core}\mc{A}+\mc{B}\n\mc{B}^{\core}\n\mc{B}=\mc{A}+\mc{B}.
  \end{eqnarray*}
  Since $\mc{A}$ is $\{1^T,3^T\}$ inverse of $(\mc{A}+\mc{B}),$ by Theorem \ref{thm3.61},
  \begin{equation*}
  \aligned
   (\mc{A}+\mc{B})^{\core}&= (\mc{A}+\mc{B})^{\#}\n (\mc{A}+\mc{B})\n\mc{X}\\
   &= \left[\left(\mc{I}-\mc{B}^{\core}\n\mc{B}\right)\n(\mc{A}^{\core})^2\n\mc{A}+(\mc{B}^{\core})^2\n\mc{B}\n\left(\mc{I}-\mc{A}^{\core}\n\mc{A}\right)\right]\n(\mc{A}+\mc{B})\n\\
   &\hspace{0.5cm}\left[\left(\mc{I}-\mc{B}\n\mc{B}^{\core}\right)\n\mc{A}^{\core}+\mc{B}^{\core}\right]\\
   &=\left[\left(\mc{I}-\mc{B}^{\core}\n\mc{B}\right)\n(\mc{A}^{\core})^2\n\mc{A}^2+(\mc{B}^{\core})^2\n\mc{B}^2\right]\n\left[\left(\mc{I}-\mc{B}\n\mc{B}^{\core}\right)\n\mc{A}^{\core}+\mc{B}^{\core}\right]\\
   &=\left[\left(\mc{I}-\mc{B}^{\core}\n\mc{B}\right)\n\mc{A}^{\core}\n\mc{A}+\mc{B}^{\core}\n\mc{B}\right]\n\left[\left(\mc{I}-\mc{B}\n\mc{B}^{\core}\right)\n\mc{A}^{\core}+\mc{B}^{\core}\right]\\
   &=\left(\mc{A}^{\core}\n\mc{A}-\mc{B}^{\core}\n\mc{B}\n\mc{A}^{\core}\n\mc{A}+\mc{B}^{\core}\n\mc{B}\right)\n\left(\mc{A}^{\core}-\mc{B}^{\core}\n\mc{B}\n\mc{A}^{\core}+\mc{B}^{\core}\right)\\
   &=\mc{A}^{\core}-\mc{B}^{\core}\n\mc{B}\n\mc{A}^{\core}+\mc{B}^{\core}\n\mc{B}\n\mc{A}^{\core}-\mc{B}^{\core}\n\mc{B}\n\mc{A}^{\core}+\mc{B}^{\core}\\
   &=\mc{A}^{\core}-\mc{B}^{\core}\n\mc{B}\n\mc{A}^{\core}\n\mc{A}+\mc{B}^{\core}=\left(\mc{I}-\mc{B}^{\core}\n\mc{B}\right)\n\mc{A}^{\core}\n+\mc{B}^{\core}.
   \endaligned
    \end{equation*}
    \vskip -0.3cm
\end{proof}

The necessity of the condition $\mc{A}\n\mc{B}=\mc{O}$ in Theorem \ref{crsum} can can verified in the following example.
\begin{example}
Consider the tensor $\mc{A}=(a_{ijkl}) \in \mathbb{R}^{(2\times3)\times (2\times3)}$ defined in Example \ref{Example1} and  a tensor $~\mc{B}=(b_{ijkl})
 \in \mathbb{R}^{\overline{2\times3}\times\overline{2\times3}}$ with entries
\begin{eqnarray*}
\mc{A}_{ij11} =
    \begin{bmatrix}
    -1 & -1 & -3 \\
    -1 & -2 &  0
    \end{bmatrix},~
\mc{A}_{ij12} = a_{ij13} =\mc{A}_{ij21} =   \mc{A}_{ij22} = \mc{A}_{ij23} =
    \begin{bmatrix}
     0 & 0 & 0\\
     0 & 0 & 0
    \end{bmatrix}.
\end{eqnarray*}
It is easy to verify $\mc{A}*_2\mc{B} \neq \mc{O}$, $\mc{A}+\mc{B} \neq \mc{O},$ and $(\mc{A}+\mc{B})^2 = \mc{O}.$
Since $\mathfrak{R}(\mc{A}+\mc{B})\neq \mathfrak{R}(\mc{A}+\mc{B})^2,$ so the core inverse of $A+B$ does not exist, whereas  $A$ and $B$ are core invertible.
\end{example}
Notice that, if $\mc{B}\n\mc{A}=\mc{O},$ then $\mc{B}\n\mc{A}^{\core}=\mc{B}\n\mc{A}\n(\mc{A}^{\core})^2=\mc{O}.$
So, with this additional assumption, we have the following corollary.
\begin{corollary}
 Let $\mc{A},~ \mc{B}\in \mathbb{C}^{I(N)\times I(N)}$ be core tensors with $\mc{A}\n\mc{B}=\mc{O}=\mc{B}\n\mc{A}$ and $\mc{A}^*\n\mc{B}=\mc{O}.$
 Then
$\left(\mc{A}+\mc{B}\right)^{\core}=\mc{A}^{\core}+\mc{B}^{\core}.$
\end{corollary}

\section{Tensor core-EP inverse}\label{SecCoreEP}

The core-EP inverse $X$ of a square matrix $A$ with $\ind{A}=k$ is defined in \cite{prasad} as the outer inverse satisfying $\R(X)=\R(A^k),\N(X)=\N((A^k)^*)$.
Also, the core-EP inverse $X$ is the unique $(1^k)$,$(2)$,$(3)$ inverse satisfying $\R(X)\subset \R(A^k)$ \cite{prasad}.

In this section, we will extend the notion of the core inverse to the core-EP inverse for square tensors.
The core-EP inverse, as the core inverse, is unique.
Some characterizations of the tensor core-EP inverse with respect to the tensor Drazin inverse, group inverse and the tensor Moore-Penrose inverse will be presented in this section.

\smallskip
The definition of the core-EP inverse of square tensors is introduced in Definition \ref{epdef}.

 \begin{definition}\label{epdef}
Let $\mc{A}\in \mathbb{C}^{I(N) \times I(N)}$ and $\ind{\mc{A}}=k.$ A tensor $\mc{X}\in \mathbb{C}^{I(N) \times I(N)}$  satisfying
\begin{enumerate}
    \item[\rm $(EP)$] $\mc{X}\n\mc{A}^{k+1}=\mc{A}^k$
    \item[\rm $(C2)$] $\mc{A}\n\mc{X}^2=\mc{X}$
    \item[\rm ($3^T$)] $(\mc{A}\n\mc{X})^*=\mc{A}\n\mc{X}$
\end{enumerate}
is called core-EP inverse of $\mc{A}$ and it is denoted as $\mc{A}^{\ep}$.
\end{definition}

Clearly, if $k=1$, then $\mc{A}^{\ep}=\mc{A}^{\core}$.

Before proving the uniqueness of the core-EP inverse, we first prove some preliminary results.
\begin{lemma}\label{lm4.2}
Let $\mc{A}\in\mathbb{C}^{I(N) \times I(N)}$ satisfy $\ind{\mc{A}}=k$.
If a tensor $\mc{X}\in\mathbb{C}^{I(N) \times I(N)}$ satisfies $(EP)$ and $(C2)$, 
then the following holds:
\begin{enumerate}
\item [\bf (a)]  $\mc{A}\n\mc{X} = \mc{A}^{m}\n\mc{X}^{m},~m\in\mathbb{N}.$
\item [\bf (b)] $\mc{X}\n\mc{A}\n\mc{X} = \mc{X}.$
\item [\bf (c)] $\mc{A}^{m}\n\mc{X}^{m}\n\mc{A}^{m} = \mc{A}^{m}$ for $ m \geq k. $
\item[\bf (d)] $\mc{X}^{m+1}\n\mc{A}^m=\mc{A}^\D$ when $m\geq k=\ind{\mc{A}}.$
\end{enumerate}

\begin{proof}
\noindent {\bf (a)} Using (C2) repetitively, we obtain
\begin{eqnarray*}
\mc{A}\n\mc{X} &=& \mc{A}\n\mc{A}\n\mc{X}^{2} =  \mc{A}^{2}\n\mc{X}^{2} = \mc{A}^{2}\n\mc{X}\n\mc{X} = \mc{A}^{2}\n\mc{A}\n\mc{X}^{2}\n\mc{X} = \mc{A}^{3}\n\mc{X}^{3}\\
&=& \cdots =  \mc{A}^{m}\n\mc{X}^{m}
\end{eqnarray*}

\noindent {\bf (b)} Again applying (EP) and (C2), we get
$$\mc{X}\n\mc{A}\n\mc{X} = \mc{X}\n\mc{A}^{m+1}\n\mc{X}^{m+1} = \mc{A}^{m}\n\mc{X}^{m+1} = \mc{A}^{m}\n\mc{X}^{m}\n\mc{X} =\mc{A}\n\mc{X}\n\mc{X} = \mc{A}\n\mc{X}^{2} = \mc{X}.$$

\noindent {\bf (c)} To claim next part let $m\geq k.$
As
$$
\aligned
\mc{A}^{m} &= \mc{A}^{k}\n\mc{A}^{m-k} \\
&= \mc{X}\n\mc{A}^{k+1}\n\mc{A}^{m-k}\\
 &= \mc{X}\n\mc{A}^{m+1} = \mc{A}\n\mc{X}^{2}\n\mc{A}^{m+1}\\
 &= \mc{A}\n\mc{X}\n\mc{X}\n  \mc{A}^{m+1} = \mc{A}^{m}\n\mc{X}^{m}\n\mc{A}^{m}.
\endaligned
$$

\noindent {\bf (d)} To show the last part, let $\mc{B} := \mc{X}^{m+1}\n\mc{A}^m$.
Then
$$\mc{A}^{m}\n\mc{B}\n\mc{A} =\mc{A}^{m}\n\mc{X}^{m+1}\n\mc{A}^m\n\mc{A}= \mc{A}^{m}\n\mc{X}^{m}\n\left(\mc{X}\n\mc{A}^{m+1}\right)= \mc{A}^{m}\n\mc{X}^{m}\n\mc{A}^{m}=\mc{A}^{m}.$$
Using the given assumption, we again obtain
$$
\aligned
\mc{B}\n\mc{A}\n\mc{B} &=\mc{X}^{m+1}\n\mc{A}^{m}\n\mc{A}\n\mc{X}^{m+1}\n\mc{A}^{m}\\
&= \mc{X}^{m+1}\n\mc{A}^{m+1}\n\mc{X}^{m+1}\n\mc{A}^{m}\\
&= \mc{X}^{m}\n\mc{A}^{m}\n\mc{X}^{m+1}\n\mc{A}^{m}\\
&=\mc{X}^{m}\n\mc{A}^{m}\n\mc{X}^{m}\n\mc{X}\n\mc{A}^{m}\\
&=\mc{X}^{m}\n\mc{A}\n\mc{X}\n\mc{X}\n\mc{A}^{m}=\mc{X}^{m}\n\mc{A}\n\mc{X}^{2}\n\mc{A}^{m}\\
&=\mc{X}^{m}\n\mc{X}\n\mc{A}^{m}=\mc{X}^{m+1}\n\mc{A}^{m}\\
&=\mc{B}
\endaligned
$$ and
$$
\aligned
\mc{A}\n\mc{B} &=\mc{A}\n\mc{X}^{m+1}\n\mc{A}^{m}= \mc{A}\n\mc{X}^2\n\mc{X}^{m-1}\n\mc{A}^{m}\\
&= \mc{X}\n\mc{X}^{m-1}\n\mc{A}^{m}=\mc{X}^{m}\n\mc{A}^{m}=\mc{X}^{m+1}\n\mc{A}^{m+1}\\
&=\mc{X}^{m+1}\n\mc{A}^{m}\n\mc{A}=\mc{B}\n\mc{A}.
\endaligned
$$
Thus $\mc{B} = \mc{X}^{m+1}\n\mc{A}^m$  is the Drazin inverse of $\mc{A}.$
 Hence, the proof is complete.
\end{proof}
\end{lemma}
The next theorem proves the uniqueness of the core-EP inverse.
\begin{theorem}\label{thm4.3}
   Let $\mc{A}\in\mathbb{C}^{I(N) \times I(N)}$ and $\ind{\mc{A}}=k.$
Then $A$ has a unique core-EP inverse in $\mathbb{C}^{I(N) \times I(N)}.$
   \end{theorem}
 \begin{proof}
 Let $\mc{X}$ and $\mc{Y}$ be two distinct core-EP inverses of $\mc{A}.$
Then\\
$\mc{X}\n\mc{A}^{k+1} = \mc{A}^k$, $\mc{A}\n\mc{X}^{2} = \mc{X}$, $ (\mc{A}\n\mc{X})^* = \mc{A}\n\mc{X}$ \\
and \\
$\mc{Y}\n\mc{A}^{k+1} = \mc{A}^k$, $\mc{A}\n\mc{Y}^{2} = \mc{Y}$, $ (\mc{A}\n\mc{Y})^* = \mc{A}\n\mc{Y}.$ \\
By Lemma \ref{lm4.2}(d),   $\mc{X}^{m+1}\n\mc{A}^m = \mc{A}^\D$ and   $\mc{Y}^{m+1}\n\mc{A}^m = \mc{A}^\D$ for any $m\geq k.$
Again by Lemma \ref{lm4.2}(c), we have
$\mc{A}^{k}\n\mc{X}^{k}\n\mc{A}^{k} = \mc{A}^{k}.$
Now

$$
\aligned
(\mc{A}^{k}\n\mc{X}^{k})^* &= (\mc{A}\n\mc{X})^*  =\mc{A}\n\mc{X} = \mc{A}^{k}\n\mc{X}^{k},
\\
\mc{X}^{k}\n(\mc{A}^{k})^2 &= \mc{X}^{k}\n\mc{A}^{k}\n\mc{A}^{k}\\
 &= \mc{X}^{k}\n\mc{X}\n\mc{A}^{k+1}\n\mc{A}^{k}= \mc{X}^{k+1}\n\mc{A}^{k+1}\n\mc{A}^{k} \\
 &= \mc{X}^{k+1}\n\mc{A}^{k}\n\mc{A}\n\mc{A}^{k} \\
 &= \mc{A}^\D\n\mc{A}\n\mc{A}^{k} =\mc{A}^{k},
\\
\mc{A}^{k}\n(\mc{X}^{k})^2 &=\mc{A}^{k}\n\mc{X}^{k}\n\mc{X}^{k} = \mc{A}\n\mc{X}^{2}\n\mc{X}^{k-1} = \mc{X}\n\mc{X}^{k-1} = \mc{X}^{k},
\endaligned
$$
and
 $(\mc{A}^{k}\n\mc{X}^{k})^* = (\mc{A}\n\mc{X})^* = \mc{A}\n\mc{X} = \mc{A}^{k}\n\mc{X}^{k}. $ This yield $\mc{X}^k$ is the core inverse of $\mc{A}^k.$ Similarly, we can show that $\mc{Y}^k$ is also the core inverse of $\mc{A}^k$.
 From the uniqueness of the core inverse, it follows $\mc{X}^k=\mc{Y}^k.$
 Now
 $$
 \aligned
 \mc{X} &= \mc{X}\n\mc{A}\n\mc{X} = \mc{X}\n(\mc{A}^{k}\n\mc{X}^{k}) \\
 &= \mc{X}\n\mc{X}\n\mc{A}^{k+1}\n\mc{X}^{k}  = \mc{X}^2\n\mc{A}^{k+1}\n\mc{X}^{k}.
  \endaligned
  $$
 Continuing in the same way, we obtain
$$
 \aligned
 \mc{X} &= \mc{X}^{k+1}\n\mc{A}^{2k}\n\mc{X}^{k} = \mc{X}^{k+1}\n\mc{A}^{k}\n\mc{A}^{k}\n\mc{X}^{k} \\
 &= \mc{A}^\D\n\mc{A}^{k}\n\mc{X}^{k} = \mc{A}^\D\n\mc{A}^{k}\n\mc{Y}^{k} \\
 &=  \mc{Y}^{k+1}\n\mc{A}^{k}\n\mc{A}^{k}\n\mc{Y}^{k} = \mc{Y}^{k+1}\n\mc{A}^{2k}\n\mc{Y}^{k} = \mc{Y}.
\endaligned
 $$
 Hence, the core-EP inverse is unique.
 \end{proof}

The existence of the core-EP inverse and its computation  can be discussed through other generalized inverses.
Theorem \ref{thm4.4} gives some characterizations of the core inverse in that direction. 
\begin{theorem}\label{thm4.4}
 Let $\mc{A}\in\mathbb{C}^{I(N)\times I(N)}$ and $\ind{\mc{A}}=k.$
 Then $\mc{A}^{\ep} = \mc{A}^\D\n\mc{A}^{l}\n\left(\mc{A}^l\right)^{\dg},$ where $l$ is an integer satisfying $l\geq k.$
Furthermore, $\mc{A}\n\mc{A}^{\ep}=\mc{A}^{l}\n\left(\mc{A}^l\right)^{\dg}.$
 \end{theorem}
  \begin{proof}
It is necessary to verify that $\mc{X} := \mc{A}^\D\n\mc{A}^{l}\n\left(\mc{A}^l\right)^{\dg}$ satisfies $(EP)$, $(C2)$ and $(3^T)$.
The equation $(EP)$ can be verified using
$$
  \aligned
  \mc{X}\n\mc{A}^{k+1} &=\mc{A}^\D\n\mc{A}^{l}\n\left(\mc{A}^l\right)^{\dg}\n\mc{A}^{k+1}\\
  &=\mc{A}^\D\n\mc{A}^{l}\n\left(\mc{A}^l\right)^{\dg}\n\mc{A}^{l}\n\left(\mc{A}^\D\right)^{l-k}\\
  &=\mc{A}^\D\n\mc{A}^{l+1}\n\left(\mc{A}^\D\right)^{l-k}\\
  &=\mc{A}^l\n\left(\mc{A}^\D\right)^{l-k}=\mc{A}^{k},
  \endaligned
  $$
 Further, $(C2)$ follows from
  $$
  \aligned
  \mc{A}\n\mc{X}^{2} &= \mc{A}\n\mc{A}^\D\n\mc{A}^{l}\n\mc({A}^{l})^{\dg}\n\mc{A}^\D\n\mc{A}^{l}\n\mc({A}^{l})^{\dg}\\
  &=\mc{A}^\D\n\mc{A}\n\mc{A}^{l}\n\mc({A}^{l})^{\dg}\n\mc{A}^{l}\n\mc{A}^\D\n\mc({A}^{l})^{\dg}\\
  &=\mc{A}^\D\n\mc{A}\n\mc{A}^{l}\n\mc{A}^\D\n\mc({A}^{l})^{\dg}\\
  &=\mc{A}^\D\n\mc{A}\n\mc{A}^\D\n\mc{A}^{l}\n\mc({A}^{l})^{\dg}\\
  &=\mc{A}^\D\n\mc{A}^{l}\n\mc{A}^{\dg}=\mc{X}.
  \endaligned
  $$
  The equation $(3^T)$ is verified as
  $$
  \aligned
  (\mc{A}\n\mc{X})^*   &= (\mc{A}\n\mc{A}^\D\n\mc{A}^{l}\n\mc({A}^{l})^{\dg})^*= (\mc{A}^\D\n\mc{A}\n\mc{A}^{l}\n\mc({A}^{l})^{\dg})^* \\
  &=  (\mc{A}^\D\n\mc{A}^{l}\n\mc{A}\n\mc({A}^{l})^{\dg})^*=  (\mc{A}^{l}\n\mc{A}^\D\n\mc{A}\n\mc({A}^{l})^{\dg})^*\\
  &= (\mc{A}^l\n\mc({A}^{l})^{\dg})^*\\
  &=\mc{A}^l\n\mc({A}^{l})^{\dg}=\mc{A}\n\mc{X}.
  \endaligned
  $$
 So, $\mc{X}$ is the core-EP inverse of $\mc{A}$.
    \end{proof}

The representations obtained in Corollary \ref{CorEPCore} generalizes \cite[Corollary 3.3]{prasad}.

	\begin{corollary}\label{CorEPCore}
		Let $\mc{A}\in\mathbb{C}^{I(N)\times I(N)}$ and $\ind{\mc{A}}=k.$
		Then
		\begin{equation}\label{EquEPCore}
		\mc{A}^{\ep} = \mc{A}^l\n \left(\mc{A}^{l+1}\right)^{\core}= \mc{A}^l\n \left(\mc{A}^{l+1}\right)^{\dg},
		\end{equation}
		where $l$ is an integer satisfying $l\geq k.$
	\end{corollary}
	\begin{proof}
		Using known representation of the core inverse from Theorem \ref{thm3.1} in conjunction with known properties of the Drazin and the group inverse, one can verify
		$$
		\aligned
		\mc{A}^l\n \left(\mc{A}^{l+1}\right)^{\core}&=\mc{A}^l\n \left(\mc{A}^{l+1}\right)^\# \n\mc{A}^{l+1}\n \left(\mc{A}^{l+1}\right)^\dg\\
		&=\mc{A}^l\n \left(\mc{A}^\D\right)^{l+1} \n\mc{A}^{l+1}\n \left(\mc{A}^{l+1}\right)^\dg \\
		&= \mc{A}^\D\n \mc{A}^{l+1} \n \left(\mc{A}^{l+1}\right)^\dg.
		\endaligned
		$$
		Now, according to the representation of the core-EP inverse from Theorem \ref{thm4.4}, it follows that
		$ \mc{A}^l\n \left(\mc{A}^{l+1}\right)^{\core}= \mc{A}^{\ep}$.
		Finally, in view of $\mc{A}^\D\n \mc{A}^{l+1}=\mc{A}^l$, it follows that $ \mc{A}^{\ep} = \mc{A}^l\n \left(\mc{A}^{l+1}\right)^\dg$.
	\end{proof}

One can observe that the proof of Theorem \ref{thm4.4} requires only $\{1,3\}$ inverse of $\mc{A}^l.$
Hence, using similar lines, we generalize the theorem as follows:
\begin{theorem}\label{thm4.5}
 Let $\mc{A}\in\mathbb{C}^{I(N)\times I(N)}$ and ind$(\mc{A})=k.$
 Then $\mc{A}^{\ep} = \mc{A}^D\n\mc{A}^{l}\n\left(\mc{A}^l\right)^{(1,3)},$ where $l$ is an integer with $l\geq k.$
 Furthermore, $\mc{A}\n\mc{A}^{\ep}=\mc{A}^{l}\n\left(\mc{A}^l\right)^{(1,3)}.$
 \end{theorem}
The core-EP inverse can be computed through the core inverse.
The next result is in this direction.
\begin{theorem}\label{thm4.6}
    Let $\mc{A}\in\mathbb{C}^{I(N)\times I(N)}$ and ind$(\mc{A})=m.$
    Then $\mc{A}$ is core-EP invertible if and only if $\mc{A}^{m}$ is core invertible.
    Furthermore,  $\left(\mc{A}^{m}\right)^{\core} = \left(\mc{A}^{\ep}\right)^{m}$  and $\mc{A}^{\ep} = \mc{A}^{m-1}\n\left(\mc{A}^{m}\right)^{\core}$
\end{theorem}

\begin{proof}
  Let $\left(\mc{A}\right)^{\ep}=\mc{X}$ and $\mc{Y}=\mc{X}^{m}.$ Now by using Lemma \ref{lm4.2}, we obtain:
  \begin{equation*}
  \aligned
  \mc{Y}\n(\mc{A}^{m})^2 &= \mc{X}^{m}\n(\mc{A}^{m})^2 = (\mc{X}^{m}\n\mc{A}^{m})\n\mc{A} = (\mc{X}^{m+1}\n\mc{A}^{m+1})\n\mc{A}=(\mc{X}^{m+1}\n\mc{A}^{m})\n\mc{A}^{m}\\
  &= \mc{A}^D\n\mc{A}^{m+1}= \mc{A}^{m+1} =\mc{A},\\
  \mc{A}^m\n\mc{Y}^2&=\mc{A}^m\n\left(\mc{X}^m\right)^2=\left(\mc{A}^m\n\mc{X}^m\right)\n\mc{X}^m=\mc{A}\n\mc{X}\n\mc{X}^m=\mc{X}^m=\mc{Y}, \mbox{ and }\\
  \left(\mc{A}^m\n\mc{Y}\right)^*&=\left(\mc{A}^m\n\mc{X}^m\right)^*=\left(\mc{A}\n\mc{X}\right)^*=\mc{A}\n\mc{X}=\mc{A}^m\n\mc{X}^m=\mc{A}^m\n\mc{Y}.
\endaligned
  \end{equation*}
  Therefore, $\mc{A}^m$ is core invertible and  $\left(\mc{A}^{m}\right)^{\core}= \mc{Y}= \left(\mc{A}^{\ep}\right)^{m}.$
  To show the converse part, let $\mc{A}^m$ be core invertible. So by Definition \ref{CoreDefT}, we have the following expressions
  $$(\mc{A}^{m})^{\core}\n\left(\mc{A}^{m}\right)^2 = \mc{A}^{m}, ~\mc{A}^{m}\n((\mc{A}^{m})^{\core})^2 = (\mc{A}^{m})^{\core}, \mbox{ and } \mc{A}^{m}\n(\mc{A}^{m})^{\core} = \mc{A}^{m}\n(\mc{A}^{m})^{\core}.$$
  Now consider $\mc{X}=\mc{A}^{m-1}\n(\mc{A}^{m})^{\core}.$
  Then we have
  $$
  \aligned
  & \mc{X}\n\mc{A}^{m+1} = \mc{A}^{m-1}\n(\mc{A}^{m})^{\core}\n\mc{A}^{m+1}\\
  &\quad = \mc{A}^{m-1}\n(\mc{A}^{m})^{\#}\n\mc{A}^{m}\n(\mc{A}^{m})^{\dg}\n\mc{A}^{m+1} = \mc{A}^{m-1}\n(\mc{A}^{m})^{\#}\n\mc{A}^{m+1} = \mc{A}^{m},\\
  &  \mc{A}\n\mc{X}^2 = \mc{A}\n\left(\mc{A}^{m-1}\n(\mc{A}^{m})^{\core}\right)^2 = \mc{A}^{m}\n(\mc{A}^{m})^{\core}\n\mc{A}^{m-1}\n(\mc{A}^{m})^{\core} \\
  &\quad = \mc{A}^{m}\n(\mc{A}^{m})^{\core}\n\mc{A}^{m-1}\n\mc{A}^{m}\n((\mc{A}^{m})^{\core})^2 = \mc{A}^{m}\n(\mc{A}^{m})^{\core}  = \mc{X},\\
  &(\mc{A}\n\mc{X})^* =  (\mc{A}\n\mc{A}^{m-1}\n(\mc{A}^{m})^{\core})^* =   (\mc{A}^{m}\n(\mc{A}^{m})^{\core})^* = \mc{A}^{m}\n(\mc{A}^{m})^{\core} = \mc{A}\n\mc{A}^{m-1}\n(\mc{A}^{m})^{\core} = \mc{A}\n\mc{X}.
  \endaligned
$$
Thus $\mc{X}$ is the core-EP inverse of $\mc{A}$ and hence  $\mc{A}^{\ep} = \mc{X} = \mc{A}^{m-1}\n(\mc{A}^{m})^{\core}.$
\end{proof}
The power of the core-EP inverse of a tensor and the core-EP of a tensor power is discussed in the next theorem.
\begin{theorem}\label{thm4.7}
 Let $\mc{A}\in\mathbb{C}^{I(N)\times I(N)}$ and $k$ be any positive integer.
 Then $A$ is core-EP invertible if and only if $\mc{A}^{k}$ is core-EP invertible.
 In particular,   $(\mc{A}^{k})^{\ep} = (\mc{A}^{\ep})^{k}$ and $\mc{A}^{\ep} = \mc{A}^{k-1}\n(\mc{A}^{k})^{\ep}.$
 \end{theorem}
 \begin{proof}
 Let ind$(\mc{A})=m$ and $\mc{A}^{\ep} = \mc{X}.$ Then by definition of core-EP inverse, we have $\mc{X}\n\mc{A}^{m+1} = \mc{A}^m,$ $\mc{A}\n\mc{X}^2 = \mc{X}$, $(\mc{A}\n\mc{X})^* = \mc{A}\n\mc{X}$.
 Choose an integer $n$ such that $0 \leq kn-m <k.$
 Now we have the following:
 \begin{equation*}
 \aligned
   &(\mc{A}^{k})^n= \mc{A}^{kn} = \mc{A}^{m}\n\mc{A}^{kn-m} = \mc{X}\n\mc{A}^{m+1}\n\mc{A}^{kn-m}=\mc{X}\n\mc{A}^{kn+1}=\mc{X}\n(\mc{A}^{k})^n\n\mc{A} \\
   &\qquad \ \ = \mc{X}^2\n(\mc{A}^{k})^n\n\mc{A}^2=\cdots=\mc{X}^k\n(\mc{A}^{k})^n\n\mc{A}^k=\mc{X}^k\n(\mc{A}^{k})^{n+1};\\
  &\mc{A}^{k}\n(\mc{X}^{k})^2 = (\mc{A}^{k}\n\mc{X}^{k})\n\mc{X}^{k} =  \mc{A}\n\mc{X}\n\mc{X}^{k} = (\mc{A}\n\mc{X}^2)\n\mc{X}^{k-1} =\mc{X}\n\mc{X}^{k-1} = \mc{X}^{k};\\
   &(\mc{A}^{k}\n\mc{X}^{k})^* = (\mc{A}\n\mc{X})^* = \mc{A}\n\mc{X} =  \mc{A}^{k}\n\mc{X}^{k}.
   \endaligned
  \end{equation*}
  Therefore $\mc{X}^{k}$  is the core-EP inverse of $\mc{A}^k$  and hence  $(\mc{A}^{k})^{\ep} = \mc{X}^{k} = (\mc{A}^{\ep})^{k}$ with ind$(\mc{A}^k)\leq n.$

Conversely, suppose ind$(\mc{A}^k)= n$ and    $(\mc{A}^{k})^{\ep} = \mc{Y}.$
So by the definition, we have $\mc{Y}\n(\mc{A}^{k})^{n+1} = (\mc{A}^{k})^n,$ $\mc{A}^{k}\n(\mc{Y})^2 = \mc{Y},$ and $(\mc{A}^{k}\n\mc{Y})^* = \mc{A}^{k}\n\mc{Y}$.
To claim the converse part, it is enough to show that $\mc{X} = \mc{A}^{k-1}\n\mc{Y}$ is the core-EP inverse of $\mc{A}.$
  The equation $(EP)$ is satisfied because of
   \begin{equation*}
   \aligned
   \mc{X}\n\mc{A}^{kn+1}&= \mc{A}^{k-1}\n\mc{Y}\n\mc{A}^{kn+1} = \mc{A}^{k-1}\n\mc{A}^k\n\mc{Y}^2\n\mc{A}^{kn+1} \\
  &= \mc{A}^{k-1}\n\mc{A}^k\n\mc{Y}\n\mc{Y}\n\mc{A}^{kn+1}
   =\mc{A}^{k-1}\n(\mc{A}^{2k}\n\mc{Y}\n\mc{Y}^2)\n\mc{A}^{kn+1}\\
   &=\cdots=\mc{A}^{k-1}\n(\mc{A}^{nk}\n\mc{Y}\n\mc{Y}^n)\n\mc{A}^{kn+1}\\
    &= \mc{A}^{kn+k-1}\n\mc{Y}^{n+1}\n\mc{A}^{kn}\n\mc{A} = \mc{A}^{kn+k-1}\n\mc{Y}^{n+1}\n(\mc{A}^{k})^{n}\n\mc{A}\\
    &= \mc{A}^{kn+k-1}\n(\mc{A}^{k})^D\n\mc{A}^{k}= (\mc{A}^{k})^D\n\mc{A}^{kn+k}=(\mc{A}^{k})^D\n(\mc{A}^{k})^2\n\mc{A}^{kn-k}\\
    &= \mc{A}^{k}\n\mc{A}^{kn-k} = \mc{A}^{kn},
    \endaligned
   \end{equation*}
  Also, $(C2)$ is satisfied:
   \begin{equation*}
   \aligned
  \mc{A}\n\mc{X}^2 &= \mc{A}\n\mc{A}^{k-1}\n\mc{Y}\n\mc{A}^{k-1}\n\mc{Y} = \mc{A}^{k}\n\mc{Y}\n\mc{A}^{k-1}\n\mc{Y}\\
  &  = \mc{A}^{k}\n\mc{Y}\n\mc{A}^{k-1}\n(\mc{A}^{k}\n\mc{Y}^2)\\
  &= \mc{A}^{k}\n\mc{Y}\n\mc{A}^{k-1}\n\left((\mc{A}^{k})^{n+1}\n\mc{Y}^{n+2}\right)  \\
  &= \mc{A}^{k}\n\left(\mc{Y}\n(\mc{A}^{k})^{n+1}\right)\n\mc{A}^{k-1}\n\mc{Y}^{n+2}\\
  &= \mc{A}^{k}\n\mc{A}^{kn}\n\mc{A}^{k-1}\n\mc{Y}^{n+2}=\mc{A}^{k-1}\n\left((\mc{A}^{k})^{n+1}\n\mc{Y}^{n+2}\right)\\
  & = \mc{A}^{k-1}\n\mc{Y} = \mc{X}.
  \endaligned
 \end{equation*}
Finally $(3^T)$ is verified by
 $$\left(\mc{A}\n\mc{X}\right)^*= \left(\mc{A}\n\mc{A}^{k-1}\n\mc{Y}\right)^* = \left(\mc{A}^{k}\n\mc{Y}\right)^*=
 \mc{A}^{k}\n\mc{Y} = \mc{A}\n\mc{X}.$$
 Thus, $\mc{X}$ is the core-EP inverse of $\mc{A}$ with ind$(\mc{A})\leq kn$.
 Therefore, $\mc{A}^{\ep}=\mc{X}=\mc{A}^{k-1}\n\mc{Y}=\mc{A}^{k-1}\n\left(\mc{A}^k\right)^{\ep}$.
 \end{proof}

\begin{theorem}\label{epofep}
   Let $\mc{A}\in\mathbb{C}^{I(N)\times I(N}$  with $\ind{\mc{A}}=m$. Then $(\mc{A}^{\ep})^{\ep}= \mc{A}^2\n\mc{A}^{\ep}$
   \end{theorem}
 \begin{proof}
 Let $\mc{A}^{\ep} = \mc{X}.$ So by definition, we have $\mc{X}\n\mc{A}^{m+1} = \mc{A}^{m}$,  $\mc{A}\n\mc{X}^2 = \mc{X}, $ and  $(\mc{A}\n\mc{X})^* =  \mc{A}\n\mc{X}.$ Let us assume $\mc{Y}= \mc{A}^2\n\mc{X}.$ To claim the result, we need to show $\mc{Y}$ is core-EP inverse of $\mc{X}.$ As $$\mc{Y}\n\mc{X}^{m+1}=\mc{A}^2\n\mc{X}^{m+2}=\mc{A}\n\mc{A}\n\mc{X}^2\n\mc{X}^{m}=\mc{A}\n\mc{X}^2\n\mc{X}^{m-1}=\mc{X}^{m},$$
 \begin{eqnarray*}
 \mc{X}\n\mc{Y}^2&=&\mc{X}\n(\mc{A}^2\n\mc{X})^2 = \mc{X}\n\mc{A}^2\n\mc{X}\n\mc{A}^2\n\mc{X} = \mc{X}\n\mc{A}^2\n\mc{X}\n\mc{A}^2\n\mc{A}\n\mc{X}^2 \\
 &=&\mc{X}\n\mc{A}^2\n\mc{X}\n\mc{A}^3\n\mc{X}^2
 =\cdots=\mc{X}\n\mc{A}^2\n\mc{X}\n\mc{A}^{m+1}\n\mc{X}^{m}\\
 &=& \mc{X}\n\mc{A}^2\n\left(\mc{A}^{m}\n\mc{X}^{m}\right) = \mc{X}\n\mc{A}^2\n\left(\mc{A}\n\mc{X}\right)=\mc{X}\n\mc{A}^3\n\mc{X}\\
 &=& \mc{X}\n\mc{A}^2\n\left(\mc{A}\n\mc{X}\right)=\mc{X}\n\mc{A}^{m+1}\n\mc{X}^{m-1} = \mc{A}^{m}\n\mc{X}^{m-1} = \mc{A}^{m}\n\mc{X}\n\mc{X}^{m-2}\\
 &=& \mc{A}^{m+1}\n\mc{X}^{m} = \mc{A}\n\left(\mc{A}^{m}\n\mc{X}^{m}\right) = \mc{A}\n\mc{A}\n\mc{X} = \mc{A}^2\n\mc{X} = \mc{Y},\mbox{ and }
  \end{eqnarray*}
  \begin{eqnarray*}
  (\mc{X}\n\mc{Y})^*&=& (\mc{X}\n\mc{A}^2\n\mc{X})^* = (\mc{X}\n\mc{A}^3\n\mc{X}^2)^* = (\mc{X}\n\mc{A}^{m+1}\n\mc{X}^{m})^*=(\mc{A}^{m}\n\mc{X}^{m})^* \\
  &=& (\mc{A}\n\mc{X})^* = \mc{A}\n\mc{X} = \mc{A}^{m}\n\mc{X}^{m}=\mc{X}\n\mc{A}^{m+1}\n\mc{X}^{m}\\
  &=& \mc{X}\n\mc{A}\n\mc{A}^{m}\n\mc{X}^{m} = \mc{X}\n\mc{A}\n\mc{A}\n\mc{X} =\mc{X}\n\mc{A}^2\n\mc{X}= \mc{X}\n\mc{Y} .
  \end{eqnarray*} Therefore, $\left(\mc{A}^{\ep}\right)^{\ep}=\mc{Y}=\mc{A}^2\n\mc{X}=\mc{A}^2\n\mc{A}^{\ep}.$
 \end{proof}
 Since $\mc{A}^{\ep}$ is of index $1$ and  $\mc{Y}\n\mc{X}^{2}=\mc{A}^2\n\mc{X}^{3}=\mc{A}\n\mc{A}\n\mc{X}^2\n\mc{X}=\mc{A}\n\mc{X}^2=\mc{X},$ we conclude as a corollary.
 \begin{corollary}\label{epcore}
 Let $\mc{A}\in\mathbb{C}^{I(N)\times I(N)}$  with ind$(\mc{A})=m.$ Then $(\mc{A}^{\ep})^{\core}= \mc{A}^2\n\mc{A}^{\ep}=\left(\mc{A}^{\ep}\right)^{\ep}.$
 \end{corollary}

 \begin{corollary}
  Let $\mc{A}\in\mathbb{C}^{I(N)\times I(N)}$  with $\ind{\mc{A}}=m$.
  Then $\left((\mc{A}^{\ep})^{\ep}\right)^{\ep}= \mc{A}^{\ep}.$
 \end{corollary}
\begin{proof}
   Let $\mc{B}= \mc{A}^{\ep}.$ By Theorem \ref{epofep}, we have
\begin{equation*}
\aligned
 \left((\mc{A}^{\ep})^{\ep}\right)^{\ep}&=\left(\mc{B}^{\ep}\right)^{\ep}=\mc{B}^2\n\mc{B}^{\ep}=\left(\mc{A}^{\ep}\right)^2\n\left(\mc{A}^{\ep}\right)^{\ep}=\left(\mc{A}^{\ep}\right)^2\n\left(\mc{A}^2\n\mc{A}^{\ep}\right)\\
 &=\left(\mc{A}^{\ep}\right)^2\n\mc{A}\n\left(\mc{A}\n\mc{A}^{\ep}\right)=\left(\mc{A}^{\ep}\right)^2\n\mc{A}\n\left(\mc{A}^m\n(\mc{A}^{\ep})^m\right)\\
 &=\mc{A}^{\ep}\n\mc{A}^{\ep}\n\mc{A}^{m+1}\n\n(\mc{A}^{\ep})^m=\mc{A}^{\ep}\n\mc{A}^{m}\n(\mc{A}^{\ep})^m=\mc{A}^{\ep}\n\mc{A}\n\mc{A}^{\ep}\\
 &=\mc{A}^{\ep}~~\mbox{ (by Lemma \ref{lm4.2}).}
 \endaligned
  \end{equation*}
 \end{proof}

%
%

 Next part of this section discusses additive property of the core-EP inverse.

 \begin{theorem}
    Let $\mc{A}\n\mc{B}=\mc{O}=\mc{B}\n\mc{A},$  and $\mc{A}^*\n\mc{B}=\mc{O}.$ Then $\left(\mc{A}+\mc{B}\right)^{\ep}=\mc{A}^{\ep}+\mc{B}^{\ep}.$
 \end{theorem}
 \begin{proof}
    Let us assume $\mc{A}\n\mc{B}=\mc{O}=\mc{B}\n\mc{A},$  and $\mc{A}^*\n\mc{B}=\mc{O}=(\mc{A}^*\n\mc{B})^*=\mc{B}^*\n\mc{A}.$
 Using these assumptions, we obtain
 $$
 \aligned
 &\mc{A}\n\mc{B}^{\ep}=\mc{A}\n\mc{B}\n\left(\mc{B}^{\ep}\right)^2=\mc{O},\\
 &\mc{B}\n\mc{A}^{\ep}=\mc{B}\n\mc{A}\n\left(\mc{A}^{\ep}\right)^2=\mc{O},\\
 & \mc{B}^{\ep}\n\mc{A}=\mc{B}^{\ep}\n\mc{B}\n\mc{B}^{\ep}\n\mc{A}=\mc{B}^{\ep}\n\left(\mc{B}^{\ep}\right)^*\n\mc{B}^*\n\mc{A}=\mc{O},\\ &\mc{A}^{\ep}\n\mc{B}=\mc{A}^{\ep}\n\mc{A}\n\mc{A}^{\ep}\n\mc{B}=\mc{A}^{\ep}\n\left(\mc{A}^{\ep}\right)^*\n\mc{A}^*\n\mc{B}=\mc{O},\\ &\mc{A}^{\ep}\n\mc{B}^{\ep}=\mc{A}^{\ep}\n\left(\mc{A}^{\ep}\right)^*\n\mc{A}^*\n\mc{B}\n\left(\mc{B}^{\ep}\right)^2=\mc{O},\\
&\mc{B}^{\ep}\n\mc{A}^{\ep}=\mc{B}^{\ep}\n\left(\mc{B}^{\ep}\right)^*\n\mc{B}^*\n\mc{A}\n\left(\mc{A}^{\ep}\right)^2=\mc{O}.
\endaligned
$$
Let $m=\max\{\mbox{ind}(\mc{A},\mc{B}\}.$
Then by Lemma \ref{lm4.2}, $\mc{A}^m\n\left(\mc{A}^{\ep}\right)^m\n\mc{A}^m=\mc{A}^m$ and $\mc{B}^m\n\left(\mc{B}^{\ep}\right)^m\n\mc{B}^m=\mc{B}^m.$
Now we see that
    \begin{equation*}
    \aligned
    (\mc{A}&+\mc{B})^m\n\left((\mc{A}^{\ep})^m+(\mc{B}^{\ep})^m\right)\n(\mc{A}+\mc{B})^m=(\mc{A}^m+\mc{B}^m)\n\left((\mc{A}^{\ep})^m+(\mc{B}^{\ep})^m\right)\n(\mc{A}^m+\mc{B}^m)\\
    &=\left(\mc{A}^m\n(\mc{A}^{\ep})^m+\mc{B}^m\n(\mc{B}^{\ep})^m\right)\n(\mc{A}^m+\mc{B}^m)\\
    &=\left(\mc{A}\n\mc{A}^{\ep}+\mc{B}\n\mc{B}^{\ep}\right)\n(\mc{A}^m+\mc{B}^m)\\
    &=\mc{A}\n\mc{A}^{\ep}\n\mc{A}^m+\mc{B}\n\mc{B}^{\ep}\n\mc{B}^m=\mc{A}^m\n(\mc{A}^{\ep})^m\n\mc{A}^m+\mc{B}^m\n(\mc{B}^{\ep})^m\n\mc{B}^m\\
    &=\mc{A}^m+\mc{B}^m,
    \endaligned
    \end{equation*}
and
    \begin{equation*}
    \aligned
    \left((\mc{A}+\mc{B})^m\n\left((\mc{A}^{\ep})^m+(\mc{B}^{\ep})^m\right)\right)^*&=\left(\mc{A}\n\mc{A}^{\ep}+\mc{B}\n\mc{B}^{\ep}\right)^*\\
    &=\left(\mc{A}\n\mc{A}^{\ep}\right)^*+\left(\mc{B}\n\mc{B}^{\ep}\right)^*=\mc{A}\n\mc{A}^{\ep}+\mc{B}\n\mc{B}^{\ep}\\
    &=(\mc{A}+\mc{B})^m\n\left((\mc{A}^{\ep})^m+(\mc{B}^{\ep})^m\right).
    \endaligned
    \end{equation*}
Thus $(\mc{A}^{\ep})^m+(\mc{B}^{\ep})^m$ is $\{1^T,3^T\}$ inverse of $(\mc{A}+\mc{B})^m.$
So, by Theorem \ref{thm4.5}, we obtain
\begin{equation*}
\aligned
(\mc{A}+\mc{B})^{\ep}&=(\mc{A}+\mc{B})^D\n(\mc{A}+\mc{B})^m\n\left((\mc{A}^{\ep})^m+(\mc{B}^{\ep})^m\right)\\
&=(\mc{A}^D+\mc{B}^D)\n(\mc{A}^m+\mc{B}^m)\n\left((\mc{A}^{\ep})^m+(\mc{B}^{\ep})^m\right)\\
&=(\mc{A}^D\n\mc{A}^m+\mc{B}^D\n\mc{B}^m)\n\left((\mc{A}^{\ep})^m+(\mc{B}^{\ep})^m\right)\\
&= \mc{A}^D\n\mc{A}^m\n\left(\mc{A}^{\ep}\right)^m+\mc{B}^D\n\mc{B}^m\n\left(\mc{B}^{\ep}\right)^m\\
&= \mc{A}^D\n\mc{A}^m\n\left(\mc{A}^{m}\right)^{\core}+\mc{B}^D\n\mc{B}^m\n\left(\mc{B}^{m}\right)^{\core}\\
&=\mc{A}^D\n\mc{A}^m\n\left(\mc{A}^{m}\right)^{(1,3)}+\mc{B}^D\n\mc{B}^m\n\left(\mc{B}^{m}\right)^{(1,3)}=\mc{A}^{\ep}+\mc{B}^{\ep}.
\endaligned
\end{equation*}
This completes the proof.
 \end{proof}

In the case of the Drazin inverse, the following result was proved in \cite{DrazBehAsJ}.
 \begin{theorem}\label{ad-draz}
    If $\mc{A}\n\mc{B}=\mc{O}=\mc{B}\n\mc{A},$ then $\left(\mc{A}+\mc{B}\right)^\D=\mc{A}^\D+\mc{B}^\D.$
 \end{theorem}

 \section{Numerical examples}\label{SecExamples}

 In this section, we present several numerical examples to verify selected properties discussed in the previous sections. All examples were carefully implemented in MATLAB.

\begin{example}\label{Exm34a}\rm
	This example is aimed to the verification of parts (a), (b), (d) and (g) of Theorem \ref{thm3.1}.
	Let $\mathcal{A}\in \mathbb{R}^{(2\times 2)\times (2\times 2)}$ with
	$$
	\aligned
	\mathcal{A}(:,:,1,1) &=10^2\cdot\bmatrix
	0.985940927109977  & 1.682512984915278\\
	1.420272484319284 &  1.962489222569553\endbmatrix,\
	\mathcal{A}(:,:,2,1) =\bmatrix
	0 &  0\\
	0 & 0 \endbmatrix,\\
	\mathcal{A}(:,:,1,2) &=10^2\cdot \bmatrix
	8.929224052859770 &  5.557379427193866\\
	7.032232245562910 &  1.844336677576532\endbmatrix,\
	\mathcal{A}(:,:,2,2) =\bmatrix
	0 & 0\\
	0 & 0\endbmatrix,
	\endaligned
	$$		
	\textbf{Verification of part (a) of Theorem} \ref{thm3.1}:
	
	We shall compute the core inverse 	$\mc{A}^{\core}$ with the formula
	$\mc{A}^{\core}=\mc{A}^{\#}\n\mc{A}\n\mc{A}^\dg$, which turns out to be $\mc{A}^{\core}=\mc{A}\n(\mc{A}^3)^\dg \n\mc{A}^2 \n\mc{A}^\dg$, since it holds	$\mc{A}^{\#}=\mc{A}(\mc{A}^3)^\dg \mc{A}$.
	We have that,
	$$
	\aligned
	\mc{A}^\dg(:,:,1,1) &=\bmatrix
	-0.002139621590719 &  0.000961949275511\\
	0       &            0\endbmatrix,\ \\
	\mc{A}^\dg(:,:,2,1) &= 10^{-3} \cdot\bmatrix
	0.154116174683626 &  0.400242571625946\\
	0            &       0\endbmatrix,\ \\
	\mc{A}^\dg(:,:,1,2) &=\bmatrix
	0.001721913469812  & 0.000005405961529\\
	0         &          0\endbmatrix,\ \\
	\mc{A}^\dg(:,:,2,2) &= \bmatrix
	0.004582706318709 & -0.000777570778470\\
	0               &    0\endbmatrix
	\endaligned	$$
	and
	$$
	\aligned
	(\mc{A}^3)^\dg(:,:,1,1) &=10^{-6} \cdot\bmatrix
	-0.169528528162183  &  0.042750108332711\\
	0       &            0\endbmatrix,\ \\
	(\mc{A}^3)^\dg(:,:,2,1) &= 10^{-7} \cdot\bmatrix
	-0.179043837563866  &  0.051654956705111\\
	0            &       0\endbmatrix,\ \\
	(\mc{A}^3)^\dg(:,:,1,2) &=10^{-7} \cdot\bmatrix
	0.864318723074744   & -0.207154961229323\\
	0         &          0\endbmatrix,\ \\
	(\mc{A}^3)^\dg(:,:,2,2) &= 10^{-6} \cdot\bmatrix
	0.280825799119855 & -0.069038738080892\\
	0               &    0\endbmatrix.
	\endaligned	$$
	Then, the core inverse of $\mc{A}$ is the following,
	$$
	\aligned
	\mc{A}^{\core}(:,:,1,1) &=\bmatrix
	-0.002139621590719 &  0.000961949275511\\
	-0.000303613542775 &  0.003332187328937\endbmatrix,\ \\
	\mc{A}^{\core}(:,:,2,1) &= 10^{-3} \cdot\bmatrix
	0.154116174683627 &  0.400242571625946 \\
	0.304669984681755 &  0.532596423124327\endbmatrix,\ \\
	\mc{A}^{\core}(:,:,1,2) &=\bmatrix
	0.001721913469812 &  0.000005405961529 \\
	0.000713871950958 & -0.001398898974183\endbmatrix,\ \\
	\mc{A}^{\core}(:,:,2,2) &= \bmatrix
	0.004582706318709 & -0.000777570778470 \\
	0.001422901360057 & -0.005026199525584\endbmatrix.
	\endaligned	$$
	Now, we shall verify the identities $(C1),\ (C2),\ (3^T)$ for the resulted core inverse  $\mc{A}^{\core}$.
	\\
	We have that
	$$
	\aligned
	\mathcal{A}^2(:,:,1,1) &=10^5\cdot\bmatrix
	1.599561492590487  & 1.100922146057665\\
	1.323212683603806  & 0.503801885212158\endbmatrix,\
	\mathcal{A}^2(:,:,2,1) &=\bmatrix
	0 &  0\\
	0 & 0 \endbmatrix,\\
	\mathcal{A}^2(:,:,1,2) &=10^5\cdot \bmatrix
	5.842677349321679 &  4.590800151195201\\
	5.176271403733929 &  2.777318467842903\endbmatrix,\
	\mathcal{A}^2(:,:,2,2) &=\bmatrix
	0 & 0\\
	0 & 0\endbmatrix.
	\endaligned
	$$		Then, we have that $\mc{A}^{\core}\n\mc{A}^2=\mc{A}$, hence $(C1)$ is true.
	\\
	Also, it holds that $$
	\aligned
	(\mathcal{A}^{\core})^2(:,:,1,1) &=10^{-4}\cdot\bmatrix
	0.214580181358990 & -0.047655377387934\\
	0.059851986462700 & -0.253852316667977\endbmatrix,\\
	(\mathcal{A}^{\core})^2(:,:,2,1) &=10^{-5}\cdot\bmatrix
	0.284712034659642 & -0.014177387864685\\
	0.108958616205944 & -0.256102470356689 \endbmatrix,\\
	(\mathcal{A}^{\core})^2(:,:,1,2) &=10^{-4}\cdot \bmatrix
	-0.099756585933128 &  0.030298875490167\\
	-0.022919369812014 &  0.131415268596208\endbmatrix,\\
	(\mathcal{A}^{\core})^2(:,:,2,2) &=10^{-4}\cdot\bmatrix
	0.339584711910359 &  0.088818590828190\\
	-0.086647284748517 &  0.423786927376011\endbmatrix,
	\endaligned
	$$
	hence, we have that $\mc{A}\n(\mc{A}^{\core})^2=\mc{A}^{\core}$ and $(C2)$ is true.
	\\
	Finally, it holds that
	$$
	\aligned
	\mc{A}\n\mathcal{A}^{\core}(:,:,1,1) &=\bmatrix
	0.647992011370661 &  0.174597600453917\\
	0.372580504169106 & -0.242482598137053\endbmatrix,\\
	\mc{A}\n\mathcal{A}^{\core}(:,:,2,1) &=\bmatrix
	0.372580504169107 &  0.248360229853187\\
	0.303348568052670 &  0.104063338661755 \endbmatrix,\\
	\mc{A}\n\mathcal{A}^{\core}(:,:,1,2) &= \bmatrix
	0.174597600453919 &  0.292718475124185\\
	0.248360229853188 &  0.338920703982751\endbmatrix,\\
	\mc{A}\n\mathcal{A}^{\core}(:,:,2,2) &=\bmatrix
	-0.242482598137050 &  0.338920703982753\\
	0.104063338661757 &  0.755940945452490\endbmatrix.
	\endaligned
	$$
	Therefore, equation $(3^T)$ is true.
	
	\textbf{Verification of part (d) of Theorem \ref{thm3.1}: }
	
	According to our calculations for part (a), it is easy to see that,
	$$
	\aligned
\mathcal{A}^{\core}\n\mc{A}^2\n\mc{A}^{\dg}(:,:,1,1)=\mc{A}^2\n\mc{A}^{\dg}\n\mathcal{A}^{\core}(:,:,1,1) &=\bmatrix
	0.647992011370665 &  0.174597600453918\\
	0.372580504169107 & -0.242482598137055\endbmatrix,\\
\mathcal{A}^{\core}\n\mc{A}^2\n\mc{A}^{\dg}(:,:,2,1)=\mc{A}^2\n\mc{A}^{\dg}\n\mathcal{A}^{\core}(:,:,2,1) &=\bmatrix
	  0.372580504169109 &  0.248360229853187\\
	0.303348568052670 &  0.104063338661752 \endbmatrix,\\
	\mathcal{A}^{\core}\n\mc{A}^2\n\mc{A}^{\dg}(:,:,1,2)=\mc{A}^2\n\mc{A}^{\dg}\n\mathcal{A}^{\core}(:,:,1,2) &= \bmatrix
	0.174597600453919 &  0.292718475124185\\
	0.248360229853188 &  0.338920703982751\endbmatrix,\\
\mathcal{A}^{\core}\n\mc{A}^2\n\mc{A}^{\dg}(:,:,2,2)=\mc{A}^2\n\mc{A}^{\dg}\n\mathcal{A}^{\core}(:,:,2,2) &=\bmatrix
	-0.242482598137050 &  0.338920703982753\\
	0.104063338661757 &  0.755940945452490\endbmatrix.
	\endaligned
	$$
Hence, from part (c) we have $\mc{A}^{\core}\n\left(\mc{A}^{\core}\right)^{\dg}=	
\left(\mc{A}^{\core}\right)^{\dg}\n\mc{A}^{\core}$ and $\mc{A}^{\core}$ is EP.	
	
\smallskip
	\textbf{Verification of part (g) of Theorem \ref{thm3.1}:}	
	
	Since, it holds	$\mc{A}^{\#}=\mc{A}(\mc{A}^3)^\dg \mc{A}$, we have that
		$$
	\aligned
	\mc{A}^\#(:,:,1,1) &=\bmatrix
	-0.005822728386101 &  0.001762848163527\\
	-0.001341208843236  & 0.007661282558445\endbmatrix,\
	\mc{A}^\#(:,:,2,1) &= \bmatrix
0& 0\\
	0            &       0\endbmatrix,\ \\
	\mc{A}^\#(:,:,1,2) &=\bmatrix
	0.009355568940286 & -0.001033016783992 \\
	0.003238756270141 & -0.009348711331342
\endbmatrix,\
	\mc{A}^\#(:,:,2,2) &= \bmatrix
	0 & 0\\
	0               &    0\endbmatrix
	\endaligned	$$
	and
	$$
\aligned
\mathcal{A}^{\core}\n\mc{A}(:,:,1,1)=\mc{A}^\#\n\mc{A}(:,:,1,1) &=\bmatrix
 1.000000000000006 & -0.000000000000000\\
0.412689678913960 & -0.817575617868222
\endbmatrix,\\
\mathcal{A}^{\core}\n\mc{A}(:,:,2,1)=\mc{A}^\#\n\mc{A}(:,:,2,1) &=\bmatrix
0 &  0\\
0 &  0 \endbmatrix,\\
\mathcal{A}^{\core}\n\mc{A}(:,:,1,2)=\mc{A}^\#\n\mc{A}(:,:,1,2) &= \bmatrix
 -0.000000000000000 &  1.000000000000001\\
0.602304320244615 &  1.645497247305027
\endbmatrix,\\
\mathcal{A}^{\core}\n\mc{A}(:,:,2,2)=\mc{A}^\#\n\mc{A}(:,:,2,2) &=\bmatrix
0 &  0\\
0 &  0\endbmatrix.
\endaligned
$$	
Hence, part (g) of Theorem \ref{thm3.1} is verified.

\end{example}

%

 \begin{example}\label{Exm41}
In this example we will verify the representation $\mc{A}^{\ep} = \mc{A}^\D\n\mc{A}^{l}\n\left(\mc{A}^l\right)^{\dg},\ l\geq \ind{\mc{A}}$, from Theorem \ref{thm4.4}.
We consider the tensor $\mathcal{A}\in \mathbb{R}^{(2\times 2)\times (2\times 2)}$ with
$$
\aligned
\mathcal{A}(:,:,1,1) &=\bmatrix
1  &0 \\
0 &  0\endbmatrix,\
\mathcal{A}(:,:,2,1) =\bmatrix
1 &  0\\
0 & 0 \endbmatrix,\\
\mathcal{A}(:,:,1,2) &= \bmatrix
0 &  0\\
0 &  0\endbmatrix,\
\mathcal{A}(:,:,2,2) =\bmatrix
0 & 1\\
0 & 0\endbmatrix.
\endaligned
$$		
We observe that  $\rra{\mathcal{A}}=2$, while $\rra{\mathcal{A}^2}=\rra{\mathcal{A}^3}=1$. Hence, $\ind{\mc{A}}=2.$
The Drazin inverse of a tensor is defined by the formula $\mc{A}^\D=\mc{A}^l\left( \mc{A}^{2l+1}\right)^\dagger \mc{A}^l.$
So, we need to verify the identities $(EP),\ (C2)$ and $(3^T)$ for the tensor defined by the formula
$\mc{A}^{\ep} = \mc{A}^2\n\left( \mc{A}^5\right)^\dagger\n \mc{A}^{4}\n\left(\mc{A}^2\right)^{\dg}.$
It holds, $(\mc{A}^2)^\dg = (\mc{A}^5)^\dg$, where
$$
\aligned
(\mc{A}^2)^\dg(:,:,1,1) &=\bmatrix
0.5  &0 \\
0.5 &  0\endbmatrix,\
(\mc{A}^2)^\dg(:,:,2,1) =\bmatrix
0 &  0\\
0 & 0 \endbmatrix,\\
(\mc{A}^2)^\dg(:,:,1,2) &= \bmatrix
0 &  0\\
0 &  0\endbmatrix,\
(\mc{A}^2)^\dg(:,:,2,2) =\bmatrix
0 & 0\\
0 & 0\endbmatrix.
\endaligned
$$	
Therefore, according to the identity $\mc{A}^{\ep} = \mc{A}^2\n\left( \mc{A}^5\right)^\dagger\n \mc{A}^{4}\n\left(\mc{A}^2\right)^{\dg},$ the
core-EP inverse of $\mc{A}$ is,
$$
\aligned
\mc{A}^{\ep}(:,:,1,1) &=\bmatrix
1 &   0\\
0 &  0\endbmatrix,\ \ \mc{A}^{\ep}(:,:,2,1) &=\bmatrix
0 &  0\\
0 &  0 \endbmatrix,\\
\mc{A}^{\ep}(:,:,1,2) &= \bmatrix
	   0 &  0 \\
0 & 0 \endbmatrix,\ \
\mc{A}^{\ep}(:,:,2,2) &=\bmatrix
0 & 0\\
0&  0\endbmatrix.
\endaligned
$$
Because of the previous calculations, the validity of $(EP),\ (C2)$ and $(3^T)$ is clear.
  \end{example}

	\begin{example}\label{Exm41Cor}
The aim of this example is to verify the representation $ \mc{A}^{\ep} = \mc{A}^l\n \left(\mc{A}^{l+1}\right)^\dg$, presented in  Corollary \ref{CorEPCore}.
The input tensor $\mathcal{A}\in \mathbb{R}^{(2\times 2)\times (2\times 2)}$ with the entries
		$$
		\aligned
		\mathcal{A}(:,:,1,1) &=\bmatrix
		2  &8 \\
		4 &  16\endbmatrix,\
		\mathcal{A}(:,:,2,1) =\bmatrix
		1 &  4\\
		2 & 8 \endbmatrix,\\
		\mathcal{A}(:,:,1,2) &= \bmatrix
		0 &  2\\
		1 &  4\endbmatrix,\
		\mathcal{A}(:,:,2,2) =\bmatrix
		0 & 1\\
		0 & 2\endbmatrix.
		\endaligned
		$$		
It holds that $\rra{\mathcal{A}}=3$ and $\rra{\mathcal{A}^2}=\rra{\mathcal{A}^3}=2$.
Hence, $\ind{\mc{A}}=2.$
		For $l=2$ we will compute the core-EP inverse of $\mc{A}$ with the formula  $\mc{A}^{\ep}=\mc{A}^l\n\left( \mc{A}^{l+1}\right)^\dagger.$
		So, we need to verify the identities $(EP),\ (C2)$ and $(3^T)$ for the tensor defined by the formula
		$\mc{A}^{\ep}=\mc{A}^2\n\left( \mc{A}^{3}\right)^\dagger.$
				\\
		We have that
		$$
		\aligned
		\mathcal{A}^2(:,:,1,1) &=\bmatrix
		 8  &  64\\
		24  & 128
\endbmatrix,\
		\mathcal{A}^2(:,:,2,1) =\bmatrix
	4  &  32\\
	12 &   64\endbmatrix,\\
		\mathcal{A}^2(:,:,1,2) &=\bmatrix
		1   & 12\\
		4   & 24
\endbmatrix,\
		\mathcal{A}^2(:,:,2,2) =\bmatrix
		  0   &  4\\
		1  &   8
\endbmatrix,
		\endaligned
		$$	
		$$
		\aligned
		\mathcal{A}^3(:,:,1,1) &=\bmatrix
		40  &  416\\
		144  & 832
\endbmatrix,\
		\mathcal{A}^3(:,:,2,1) =\bmatrix
		20  &  208\\
		72 &   416\endbmatrix,\\
		\mathcal{A}^3(:,:,1,2) &=\bmatrix
		6   & 72\\
		24   & 144
\endbmatrix,\
		\mathcal{A}^3(:,:,2,2) =\bmatrix
		1   &  20\\
		6  &   40
\endbmatrix
		\endaligned
		$$	
		and
		$$
		\aligned
		(\mc{A}^3)^\dg(:,:,1,1) &=\bmatrix
	0.048094852307653 & -0.270803595540296\\
	0.024047426153826 & -0.276815452078752
\endbmatrix,\\
		(\mc{A}^3)^\dg(:,:,2,1) &=\bmatrix
		0.047143242930443 & -0.264910690173993\\
		0.023571621465222 & -0.270803595540299 \endbmatrix,\\
		(\mc{A}^3)^\dg(:,:,1,2) &= \bmatrix
		  -0.003806437508841 &  0.023571621465222\\
		-0.001903218754421 &  0.024047426153827
\endbmatrix,\\
		(\mc{A}^3)^\dg(:,:,2,2)& =\bmatrix
		 -0.007612875017682 &  0.047143242930443\\
		-0.003806437508841 &  0.048094852307653
\endbmatrix.
		\endaligned
		$$	
		Therefore, the core-EP inverse $\mc{A}^{\ep} = \mc{A}^2\n\left( \mc{A}^3\right)^\dagger$ of $\mc{A}$ is equal to
		$$
		\aligned
		\mc{A}^{\ep}(:,:,1,1) &=\bmatrix
	   0.210144927536232 & -0.509316770186336\\
	0.082815734989648 & -1.018633540372671
\endbmatrix,\\ \mc{A}^{\ep}(:,:,2,1) &=\bmatrix
	 0.206521739130437 & -0.490683229813671\\
	0.083850931677020 & -0.981366459627341
\endbmatrix,\\
		\mc{A}^{\ep}(:,:,1,2) &= \bmatrix
	-0.014492753623189 &  0.074534161490684\\
	0.004140786749482 &  0.149068322981367
 \endbmatrix,\\
		\mc{A}^{\ep}(:,:,2,2) &=\bmatrix
		 -0.028985507246377 &  0.149068322981367\\
		0.008281573498965 &  0.298136645962735
\endbmatrix.
		\endaligned
		$$
		Now, we shall verify the identities $(EP),\ (C2)$ and $(3^T)$ for the resulted core-EP inverse  $\mc{A}^{\ep}$.
		\\
		We have that
		$$
		\aligned
		\mc{A}^{\ep}\mathcal{A}^3(:,:,1,1) &=\bmatrix
		8  &  64\\
		24  & 128
\endbmatrix,\
		\mc{A}^{\ep}\mathcal{A}^3(:,:,2,1) =\bmatrix
		4  &  32\\
		12 &   64\endbmatrix,\\
		\mc{A}^{\ep}\mathcal{A}^3(:,:,1,2) &=\bmatrix
		1   & 12\\
		4   & 24
\endbmatrix,\
		\mc{A}^{\ep}\mathcal{A}^3(:,:,2,2) =\bmatrix
		0   &  4\\
		1  &   8
\endbmatrix,
		\endaligned
		$$			
Then, we have that $\mc{A}^{\ep}\n\mc{A}^3=\mc{A}^2$, hence $(EP)$ is true.
		\\
		Also, it holds that $$
		\aligned
		(\mathcal{A}^{\ep})^2(:,:,1,1) &=\bmatrix
	0.098171152518979 & -0.337474120082818\\
	0.013802622498275 & -0.674948240165636
\endbmatrix,\\
		(\mathcal{A}^{\ep})^2(:,:,2,1) &=\bmatrix
	0.096273291925468 & -0.329192546583857\\
	0.013975155279503 & -0.658385093167713
\endbmatrix,\\
		(\mathcal{A}^{\ep})^2(:,:,1,2) &= \bmatrix
			  -0.007591442374051  & 0.033126293995860\\
		0.000690131124914  & 0.066252587991719
\endbmatrix,\\
		(\mathcal{A}^{\ep})^2(:,:,2,2) &=\bmatrix
		  -0.015182884748102  & 0.066252587991719\\
		0.001380262249827 &  0.132505175983439
\endbmatrix.
		\endaligned
		$$
Hence, we have that $\mc{A}\n(\mc{A}^{\ep})^2=\mc{A}^{\ep}$ and $(C2)$ is satisfied.
		\\
		Finally, it holds that
		$$
		\aligned
		\mc{A}\n\mathcal{A}^{\ep}(:,:,1,1) &=\bmatrix
		0.503105590062112 & -0.024844720496894\\
	0.496894409937889  &-0.049689440993788
\endbmatrix,\\
		\mc{A}\n\mathcal{A}^{\ep}(:,:,2,1) &=\bmatrix
		0.496894409937894 &  0.024844720496893\\
		0.503105590062117 &  0.049689440993786
 \endbmatrix,\\
		\mc{A}\n\mathcal{A}^{\ep}(:,:,1,2) &= \bmatrix
				-0.024844720496895 &  0.198757763975156\\
		0.024844720496894  & 0.397515527950311
\endbmatrix,\\
		\mc{A}\n\mathcal{A}^{\ep}(:,:,2,2) &=\bmatrix
		-0.049689440993790 &  0.397515527950311\\
		0.049689440993788 &  0.795031055900622
\endbmatrix.
		\endaligned
		$$
		Therefore, the equation $(3^T)$ is true.
\end{example}

\section{Conclusion}\label{SecConclusion}

So far, the core and core-EP inverses in the matrix case are introduces in different ways.
Our research introduces the core and the core-EP inverses of complex square tensors by means of specific definitions.
Additionally, some of their characterizations, representations and properties are investigated.
The results are verified using specific algebraic approach, based on proposed definitions and previously verified properties.
The approach used here is new even in the matrix case.

Illustrative numerical examples are presented.

In order to finalize our conclusion, it will be useful to mention some possibilities for further research.\\
1. According to Theorem \ref{thm3.1}(e), the tensor core inverse $\mc{A}^{\core}$ is EP.
Therefore, it is reasonable to expect that $\mc{A}^{\core}$ inherits all properties of EP matrices stated in \cite{TianEP}.
These properties would be an interesting topic for further research.\\
2. It will be interesting to continue the results from \cite{PappasAlu} and investigate properties of the $\lambda$-Aluthge transform of the core inverse in both the matrix and the tensor case.\\
3. Also, the reverse order law of multiple tensor products would be an interesting area for further research.

\medskip
\noindent{\bf{Acknowledgments.}}\\
Ratikanta Behera acknowledges the support provided by Science and Engineering Research Board (SERB), Department of Science and Technology, India, under the Grant No. EEQ/2017/000747.\\
Predrag Stanimirovi\'c gratefully acknowledges support from the Ministry of Education, Science and Technological Development, Republic of Serbia, Grant No. 174013.
\bibliographystyle{abbrv}
\bibliographystyle{vancouver}

\begin{thebibliography}{10}

\bibitem{baks} O.~M. Baksalary, G.~Trenkler,
{\em Core inverse of matrices}, Linear Multilinear Algebra, {\bf 58} (2010), 681--697.

\bibitem{baks1}
O.~M. Baksalary and G.~Trenkler.
{\em  On a generalized core inverse}, Appl. Math. Comput. {\bf 236} (2014), 450--457.

\bibitem{Behera} R. Behera and D. Mishra,
{\em Further results on generalized inverses of tensors via the Einstein product}, Linear Multilinear Algebra {\bf 65} (2017), 1662--1682.

\bibitem{DrazBehAsJ}
R.~Behera, A.~K. Nandi, J.~K. Sahoo,
{\em Further results on the Drazin inverse of even-order tensors}, Submitted, arxiv.org/abs/1904.10783.


\bibitem{BrazellT}
M. Brazell, N. Li, C. Navasca, C. Tamon, {\em Solving multilinear systems via tensor inversion}, SIAM J. Matrix Anal. Appl. {\bf 34} (2013), 542--570.

\bibitem{chen2009}
J.~L. Chen, G.~F. Zhuang, Y.~M. Wei.
{\em The {D}razin inverse of a sum of morphisms}, Acta Math. Sci. Ser. A (Chin. Ed.) 29(3):538--552, 2009.

\bibitem{draz}
M.~P. Drazin,
{\em Pseudo-inverses in associative rings and semigroups}, Amer. Math. Monthly {\bf 65} (1958), 506--514.

\bibitem{ein}
A.~Einstein, {\em The foundation of the general theory of relativity}, Annalen der Physik {\bf 49} (1916), 769--822.


\bibitem{ferreyra}
D.~E. Ferreyra, F.~E. Levis, and N.~Thome.
{\em Revisiting the core {EP} inverse and its extension to rectangular matrices}, Quaest. Math. {\bf 41} (2018), 265--281.

\bibitem{gao2018}
Y.~Gao, J.~Chen.
{\em Pseudo core inverses in rings with involution}, Comm. Algebra {\bf 46} (2018), 38--50.

\bibitem{gao2019}
Y.~Gao, J.~Chen, P.~Patricio,
{\em Weighted core-ep inverse of a rectangular matrix}, arXiv preprint arXiv:1803.10186, 2018.

\bibitem{Gao} Y. F. Gao, J.L.  Chen and P. Patricio,
{\em Continuity of the core-EP inverse and its applications}, 	arXiv:1805.08696, 2018.

\bibitem{JiT} J. Ji, Y. Wei,
{\em Weighted Moore-Penrose inverses and fundamental theorem of even-order tensors with Einstein product}, Front. Math. China. {\bf 12} (2017), 1319--1337.

\bibitem{Ji18} J. Ji, Y. Wei,
{\em The Drazin inverse of an even-order tensor and its application to singular tensor equations}, Comput. Math. Appl. {\bf 75} (2018), 3402--3413.

\bibitem{kurata} H.~Kurata,
{\em Some theorems on the core inverse of matrices and the core partial ordering}, Appl. Math. Comput. {\bf 316} (2018), 43--51.

\bibitem{lai}
W.~Lai, D.~Rubin, E.~Krempl,
{\em Introduction to Continuum Mechanics}, Butterworth Heinemann, Oxford, 2009.

\bibitem{LiangT} M. Liang, B. Zheng, R. Zhao,
{\em Tensor inversion and its application to the tensor equations with Einstein product}, Linear Multilinear Algebra {\bf 67} (2019), 843--870.

\bibitem{Ma18} H. Ma,
{\em Optimal perturbation bounds for the core inverse}, Appl. Math. Comput. {\bf 336} (2018),  176--181.

\bibitem{Ma19} H. Ma,
{\em A characterization and perturbation bounds for the weighted core-EP inverse}, Quaestiones Mathematicae, 2019, DOI: 10.2989/16073606.2019.1584773

\bibitem{Ma191} H. Ma, P.S. Stanimirovi\'c,
{\em Characterizations, approximation and perturbations of the core-EP inverse}, Appl. Math. Comput., https://doi.org/10.1016/j.amc.2019.04.071.

\bibitem{Panigrahy} K. Panigrahy, D. Mishra,
{\em An extension of the Moore-Penrose inverse of a tensor via the Einstein product}, arXiv:1806.03655v1, 10 Jun 2018.

\bibitem{PappasAlu} D. Pappas, V.N. Katsikis, P.S. Stanimirović,
{\em The $\lambda$-Aluthge transform of EP matrices}, Filomat {\bf 32} (2018), 4403--4411.

\bibitem{prasad}
K.~M.~Prasad, K.~S. Mohana,
{\em Core-{EP} inverse}, Linear Multilinear Algebra {\bf 62} (2014), 792--802.

\bibitem{MD} K.M. Prasad, M.D. Raj,
{\em Bordering method to compute core-EP inverse}, Special Matrices, {\bf 6} (2018),  193--200.

\bibitem{CoreIter} K.M. Prasad, M.D. Raj, M. Vinay,
{\em Iterative method to find core-EP inverse}, Bulletin of Kerala Mathematics Association, Special Issue Vol. 16, No. 1 (2018), 139--152.

\bibitem{rakic}
D.~S. Raki\'{c}, N.~\v C. Din\v{c}i\'{c}, D.~S. Djordjevi\'{c},
{\em Group, {M}oore-{P}enrose, core and dual core inverse in rings with involution}, Linear Algebra Appl. {\bf 463} (2014), 115--133.

\bibitem{rakicAMC} D.~S. Raki\'{c}, N.~\v C. Din\v{c}i\'{c}, D.~S. Djordjevi\'{c},
{\em Core inverse and core partial order of Hilbert space operators} Appl. Math. Comput. {\bf 244} (2014), 283--302.

\bibitem{stan}
P.~Stanimirovic, M.~Ciric, V.~Katsikis, C.~Li, H.~Ma.
{\em Outer and (b,c) inverses of tensors}, Linear Multilinear Algebra, https://doi.org/10.1080/03081087.2018.1521783.

\bibitem{SunT} L. Sun, B. Zheng, C. Bu, Y. Wei,
{\em Moore-Penrose inverse of tensors via Einstein product}, Linear Multilinear Algebra {\bf 64} (2016), 686--698.

\bibitem{TianEP} Y. Tian, H. Wang,
{\em Characterizations of EP matrices and weighted-EP matrices}, Linear Algebra Appl. {\bf 434} (2011), 1295--1318.

\bibitem{wang2015}
H.~Wang, X.~Liu,
{\em Characterizations of the core inverse and the core partial ordering}, Linear Multilinear Algebra {\bf 63} (2015), 1829--1836.

\bibitem{xu2017}
S.~Xu, J.L. Chen, X.~Zhang,
{\em New characterizations for core inverses in rings with involution}, Front. Math. China {\bf 12} (2017), 231--246.

\bibitem{ZhouLimit} M.M. Zhou, J.L. Chen, T.T. Li, D.G. Wang,
{\em Three limit representations of the core-EP inverse}, Filomat {\bf 32} (2018), 5887--5894.

\end{thebibliography}

\end{document}